\documentclass{agn_article}
\usepackage{agn_all}

\usepackage[hidelinks]{hyperref}
\usepackage[percent]{overpic}
\usepackage{subcaption}
\usetikzlibrary{arrows.meta,positioning,calc}

\renewcommand{\Symp}{\Sym^{++}}
\newcommand{\TBiot}{T_{\rm Biot}}
\newcommand{\TBiotTilde}{\Ttilde_{\rm Biot}}
\newcommand{\tBiot}{t_{\rm Biot}}
\newcommand{\tBiotTilde}{\ttilde_{\rm Biot}}

\renewcommand{\todo}[1][]{\textcolor{blue}{TODO\ifx\empty#1\empty\else: #1\fi}}

\allowdisplaybreaks[1]
\title{Vector and matrix monotonicity characterisation of the ``load forces'' in nonlinear elasticity: compressible versus incompressible response}
\knownauthors[ghiba]{vossd,ghiba,martin_e,neff}
\date{\today}

\begin{document}
	\maketitle
	\begin{abstract}
		\noindent In this paper, we study the relation between different notions of monotonicity for isotropic hyperelastic compressible and incompressible materials. We study how to define the monotonicity concept on the manifold
		$\rm{Sym}^{++}(n)\cap \rm{SL}(n)$
		of positive definite symmetric tensors with unit determinant and the corresponding spectral manifold $\mathcal S=\left\{(\lambda_1,\dotsc,\lambda_n)\in \mathbb{R}^n_+\,\middle|\, \lambda_1\.\lambda_2\cdots\lambda_n=1\right\}$, and give a precise characterisation of monotonicity in the case of hyperelastic materials. In particular, we consider the monotonicity properties of the Biot stress tensor $T_{\rm Biot}(U)$, the engineering stress, or the load tensor.
	\end{abstract}
	\noindent {\bf MSC 2020 subject classification: 74B20, 74A20, 74A10}
	{\parskip=-0.4mm\tableofcontents}
	\section{Introduction}
	Convexity conditions on the elastic energy as well as monotonicity or invertibility conditions in certain stress-strain pairs play an important role in the modeling of nonlinear
	elasticity. They are used as constitutive restrictions implying a number
	of useful properties \cite{ball1976convexity,agn_voss2026monotonicity,Ogden83,Hill68,hill1968constitutiveb,hill1970,agn_wollner2026concurrent,agn_ghiba2026polyconvexity,agn_voss2019volIsLog}.
	While, a priori, it is not always clear in which variables convexity or
	monotonicity should be required, for compressible elastic materials we focus here on the convexity of
	the elastic energy $U\mapsto \What(U)$, in terms of the right stretch tensor $U = \sqrt{F^T F}\in \Symp(n)$, or the matrix-monotonicity of the symmetric Biot stress tensor $\TBiot(U) \colonequals \D_U \What(U)=\det U\.R^T\.\sigma\.R\.U^{-1}\in \Sym(n)$,
	in the sense that\footnote{
		We say that $\What$ is differentiable at $U\in \Symp(n)$
		if there exists a linear mapping $\D_U \What(U)\col \Sym(n)\to\R$ such that $\D_U\What(U).H = \lim_{s\to 0} \frac{\What(U+s\.H)-\What(U)}{s} \ \forall\; H\in \Sym(n)\,.$ The Biot stress tensor satisfies $\iprod{\TBiot(U)\,, H}=\D_U\What(U).H$, for all $U\in \Symp(n)$, $H\in \Sym(n)$.
	}
	\begin{align}
		\label{cc}
		\iprod{\TBiot(U)-\TBiot(\Ubar)\,,
			U-\Ubar}>0
		\qquad
		\forall \; U,\Ubar \in \Symp(n)\,,\quad {U\neq \Ubar},
	\end{align}
	where $\sigma$ is the Cauchy stress tensor and $\Symp(n)$ denotes the set of positive definite symmetric $n\times n$ matrices.
	The energy $U\mapsto \What(U)$ represents the reformulation of an isotropic elastic energy $W\col \GLp(n)\to \R_+$ in terms of the stretch tensor $U$, i.e.\ $W(F)=\What(U).$
	In isotropic nonlinear elasticity, the Biot stress corresponds to the ``loads'', and \eqref{cc} immediately implies injectivity of the load--stretch relation in terms of the principal stretches $\lambda_i=\lambda_i(\sqrt{F^TF})$.
	For compressible response, condition \eqref{cc} excludes, for instance,
	certain material instabilities like the Rivlin cube bifurcation \cite{agn_ghiba2025biot} and it is implicitly used in ADINA$^{\text{TM}}$ as a stability marker \cite{baaser2026hyperelastic}. The geologist G.F. Becker has already implicitly imposed a monotonicity condition in terms of $\log U$ in his 1893 attempt to derive an ideal isotropic elasticity law \cite{becker1893,agn_neff2014rediscovering}. He proposed
	\begin{align}
		\TBiot(U)=\widetilde{T}_{\rm Biot}(U)=2\.\mu\log U+\lambda\.\tr(\log U)\.\id\,,
	\end{align}
	not being aware that this constitutive law is not hyperelastic in general.
	It is worth noting that the convexity of $U\mapsto \What(U)$ is not, in general, in conflict with polyconvexity and $\What(U)\to +\infty$ for $\det F\to 0$, as the example \cite{agn_wollner2026concurrent,agn_ghiba2026polyconvexity}
	\begin{align}
		\label{clog}
		W(F)=\mu\left(\frac{1}{2}\|F\|^2-\log \det F\right)=\mu\left(\frac{1}{2}\|U\|^2-\log \det U\right)=\mu\left(\frac{1}{2}\| e^{\log U}\|^2-\tr \log U\right)
	\end{align}
	shows. Indeed, the term $-\log \det U$ is convex in $U$ on $\Symp(n)$ \cite{agn_lehmich2013convexity}, and the Biot stress derived from \eqref{clog} satisfies therefore \eqref{cc}. In many cases, however, one considers
	slightly compressible response, for example in the modelling of rubber.
	If one assumes that $\What(U)$ grows large as $\det F$ moves away from
	1, then neither convexity of $\What$ with respect to $U$, nor
	monotonicity of $\TBiot$ with respect to $U$, can in general be
	expected \cite{ball1976convexity}, see, e.g.\ Figure \ref{CNP}.
	Therefore, both convexity and monotonicity notions alluded to above are inappropriate if
	considered directly for the incompressible case $\det F = 1 \ {(\det U=1)}.$ In this contribution, therefore, we will give an answer on how to deal with
	incompressibility in conjunction with these convexity and monotonicity concepts.
	In order to do that, we start with the unconstrained, compressible case.
	Since we are dealing with isotropic nonlinear elasticity, it is useful to
	consider the relation between vector and matrix monotonicity first. It then becomes
	evident that, on the constraint set $\det F=1$, a naive transfer of monotonicity
	concepts is not possible\footnote{This is one of the reasons why the Coleman--Noll inequality \cite{coleman1959thermostatics} is not regarded as a useful general constitutive condition in the incompressible setting.} and one needs to resort to intrinsic notions.
	We carry out this program below, leading to a preference for constitutive conditions expressed in the logarithmic strain tensor $\log U$, see, for instance, \cite{agn_neff2015exponentiatedI,neff2024hypo,agn_neff2015geometry}. The decisive advantage of the logarithmic representation is that the constraint set
	\begin{align}
		\det F = 1
		\qquad \iff \qquad \tr(\log U)=0
	\end{align}
	is a convex (linear) set in logarithmic strain space and therefore allows the standard local--global monotonicity arguments to be transferred from the compressible to the incompressible response. The logarithmic monotonicity condition considered below is, however, a different constitutive requirement and not the direct restriction of classical Biot monotonicity with respect to $U$.
	Let us briefly return to the question whether the Coleman--Noll inequality \cite{coleman1959thermostatics}, which implies \eqref{cc}, is suitable for slightly compressible materials.
	To this end, assume that $\What$ is still defined on $\Symp(n)$ but grows large when $\det U$ deviates slightly from $1$. Then the resulting energy will not satisfy that $U\mapsto \What(U)$ is convex on $\Symp(n)$, see Figure \ref{CNP}.\footnote{Already Hill \cite[p. 457]{hill1970} has observed that the Coleman-Noll inequality would rule out the incompressible Neo-Hooke formulation.}
	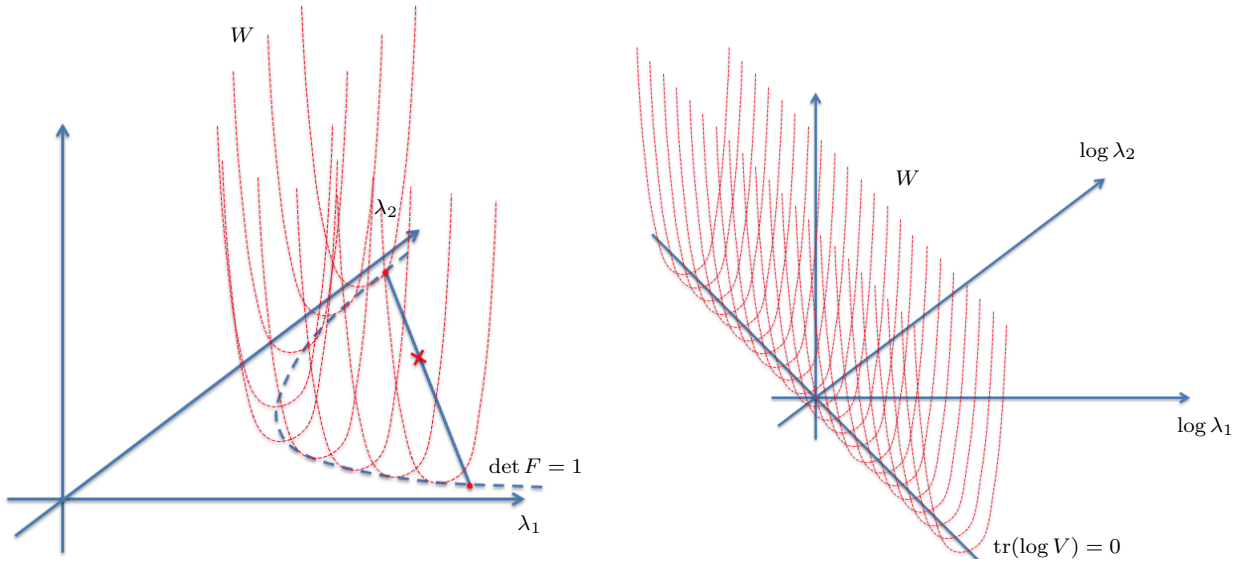
\begin{figure}[h!]
		\centering
		\begin{subfigure}[t]{0.47\linewidth}
			\centering
			\begin{overpic}[width=\linewidth]{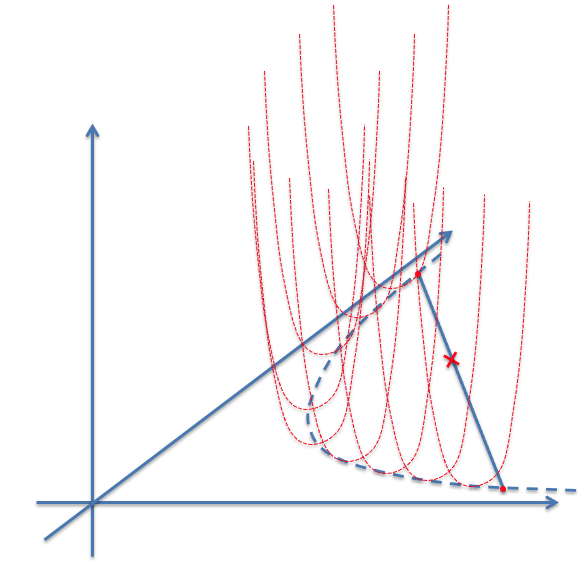}
				\put(95,5){\footnotesize $\lambda_1$}
				\put(70,60){\footnotesize $\lambda_2$}
				\put(45,90){\footnotesize $W$}
				\put(90,15){\footnotesize $\det F=1$}
			\end{overpic}
			\caption{Convexity properties may change dramatically under a nonlinear side-condition like $\det F=1$ as shown in the picture: A non-negative function which vanishes exactly for $\det F=\lambda_1\lambda_2=1$ cannot be convex.}
			\label{CNP}
		\end{subfigure}
		\hfill
		\begin{subfigure}[t]{0.47\linewidth}
			\centering
			\begin{overpic}[width=\linewidth]{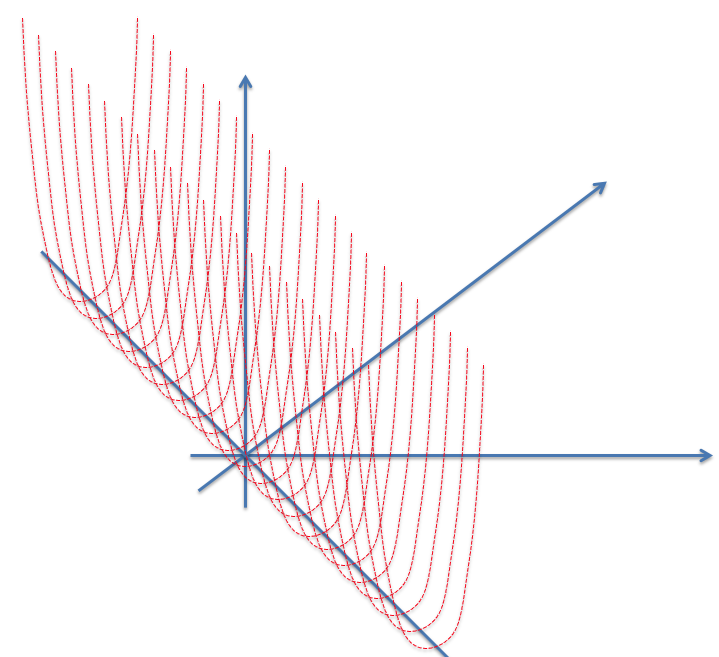}
				\put(97,22){\footnotesize $\log \lambda_1$}
				\put(80,70){\footnotesize $\log \lambda_2$}
				\put(48,65){\footnotesize $W$}
				\put(65,1){\footnotesize $\tr(\log V)=0$}
			\end{overpic}
			\caption{The function $W_{\rm eH}$ is convex in $\log V$ under the nonlinear side-condition $\det F=1$, since $\det F=1 \iff \tr(\log V)=0$.}
			\label{CNPlog}
		\end{subfigure}
		\caption{Comparison between convexity in the singular values and convexity in logarithmic strain variables under the incompressibility constraint.}
		\label{fig:CNP-comparison}
	\end{figure}
	
	A convexity condition on the map $\log V\mapsto \Wtilde(\log V)$ with $\Wtilde(X)=\What(e^X)$ does not suffer from the same problem, since the side condition $\det V=1$ is equivalent to $\tr(\log V)=0$, which defines a linear set, and we may easily write down an energy convex in $\log V$ which grows large outside $\det F=1$, e.g.\ \cite{agn_neff2015exponentiatedI}
	\begin{align}
		W_{\rm eH}(\log V)\colonequals \frac{\mu}{k}\.e^{k\.\|\dev_3\log\.V\|^2}+\frac{\kappa}{2\. \khat}\.e^{\khat\. {[\tr(\log V)]^2}}.
	\end{align}
	This energy is convex in $\log V$ for $k,\khat>0$, see Figure \ref{CNPlog}.
	Let us specify the meaning we use for the Biot stress for incompressible materials. We start by considering an isotropic energy $W\col \GLp(n)\to \R_+$, $W=W(F)$ and its representation $\What(U)=W(F)$ in terms of the right stretch tensor. Then we compute the $\TBiot$-stress using the formula for compressible materials and, finally, define its restriction to $\mathcal M=\{U\in\Symp(n)\mid \det U=1\}$ by
	\begin{align}
		\TBiot^{\rm ext,0}\col \mathcal M\to \Sym(n)\,, \qquad \quad \TBiot^{\rm ext,0}(U)\colonequals \D_U\What(U)\in\Sym(n) \quad \forall \; U\in \mathcal M\,.
	\end{align}
	The superscript $\mathrm{ext},0$ indicates that this is the ambient, zero-pressure representative obtained by first differentiating the energy on $\Symp(n)$ and then restricting the resulting tensor field to $\mathcal M$. The tensor $\TBiot^{\rm ext,0}(U)$ itself is not tangent to $\mathcal M$; rather, its action on tangent directions represents the differential of the constrained energy. However, $\TBiot^{\rm ext,0}$ depends on the chosen ambient extension of the incompressible energy and is therefore not, by itself, an intrinsic constitutive quantity of the constrained material. Two such ambient derivatives can differ on $\mathcal M$ by a normal contribution $c(U)\.U^{-1}$.
	An extension-independent representative is obtained by requiring the associated Kirchhoff stress to be deviatoric, namely
	\begin{align}
		\TBiot^{\rm inc}(U)\colonequals \D_U\What(U)-\frac{1}{n}\tr\bigl(\D_U\What(U)\,U\bigr)U^{-1}.
	\end{align}
	We refer to Appendix \ref{eiA} for a detailed discussion of intrinsic versus extrinsic Biot stresses and the extension-independence of $\TBiot^{\rm inc}$. The full extrinsic Biot-stress tensor usually defined in the incompressible case is, cf.\ \cite[Eq. (2.81)]{ogden2001elements},
	\begin{align}
		\TBiot^{\rm ext}(U)\colonequals \D_U[ \What(U)-p\.(\det U-1)]=\underbrace{\D_U\What(U)}_{\rm constitutive \ stress}-\underbrace{p\. U^{-1}}_{\rm reaction-stress}\in\Sym(n)
	\end{align}
	where $p$ is an arbitrary scalar Lagrange multiplier to be determined from the equilibrium equations and the boundary conditions. The zero-pressure and full extrinsic representatives satisfy
	\begin{align}
		\iprod{\TBiot^{\rm ext,0}(U)\,, H}=\D_U\What(U).H=\iprod{\TBiot^{\rm ext}(U)\,, H}, \quad & \text{for all} \quad U\in \mathcal M \\& \text{and} \quad H\in T_U\mathcal M=\{H\in\Sym(n)\.|\ \tr(U^{-1}H)=0\}.\notag
	\end{align}
	We use the term extrinsic for quantities defined using energies on the entire ambient space $\Sym(n)$ or on its open subspaces. In the following, $\TBiot^{\rm ext,0}$ denotes the \emph{zero-pressure extrinsic Biot-stress tensor for incompressible materials}, while $\TBiot^{\rm ext}$ denotes the full \emph{extrinsic Biot-stress tensor for incompressible materials}. The intrinsic representative $\TBiot^{\rm inc}$ is defined above and discussed further in Appendix \ref{eiA}.
	It turns out that, while in the compressible case the monotonicity of the Biot stress $\TBiot$ is defined by the inequality \eqref{cc}, in the incompressible case we consider the following (non-equivalent) logarithmic monotonicity requirement for the zero-pressure extrinsic Biot-stress tensor:
	\begin{align}
		\label{logW}
		\iprod{\TBiot^{\rm ext,0}(U)-\TBiot^{\rm ext,0}(\Ubar)\,, \log U -\log \Ubar} > 0, \quad & \text{for all}\quad U, \Ubar\in \Symp(n)\,, \ \ U\neq \Ubar \ \notag \\& \text{with} \ \ \tr(\log U)=0=\tr(\log\Ubar)\,.
	\end{align}
	The pairing in \eqref{logW} is a mixed stress--strain pairing: the Biot stress is energetically conjugate to $U$, whereas the corresponding Kirchhoff stress is naturally conjugate to $\log U$. Thus, \eqref{logW} is motivated by the logarithmic parametrization of the constraint rather than by energetic conjugacy.
	For a prescribed pressure selection $p=p(U)$, it is mathematically meaningful to consider the same definition for the full extrinsic Biot-stress tensor, i.e.
	\begin{align}
		\label{logWf}
		\iprod{\TBiot^{\rm ext}(U)-\TBiot^{\rm ext}(\Ubar)\,, \log U -\log \Ubar} > 0, \quad & \text{for all}\quad U, \Ubar\in \Symp(n)\,, \ \ U\neq \Ubar \ \notag \\& \text{with} \ \ \tr(\log U)=0=\tr(\log\Ubar)\,,
	\end{align}
	but in this condition the unknown Lagrange multiplier will be involved. Therefore, such a monotonicity requirement does not depend only on the material parameters and is not a purely constitutive requirement. Without fixing the pressure selection, $\TBiot^{\rm ext}$ represents a family of stress tensors rather than a uniquely defined tensor-valued mapping.

	From the entire class of possible extrinsic Biot-stress tensors, we choose a special one, i.e.\ plane-stress
	\begin{align}
		\TBiot^{\rm ps}(U)\colonequals\D_U\What(U)-\iprod{\D_U\What(U)\.U,N_3} U^{-1},
	\end{align}
	with $N_3$ denoting the projector associated with the distinguished third principal direction.\footnote{See the spectral decomposition of $U=\sum_{i=1}^3\lambda_i N_i, \ N_i=q_i\otimes q_i.$}
	Since, for incompressibility, $\tr(\log U)=0$, the monotonicity of the Biot-stress tensors can in some situations be characterized as follows. For incompressible nonlinear elasticity, the convexity of the reduced energy potential is sometimes considered \cite{anssari2023large,anssari2022modelling}.
	More precisely, consider $n=3$ and $\What(U)=g(\lambda_1,\lambda_2,\lambda_3)$. Since $\lambda_3=\frac{1}{\lambda_1\lambda_2}$, one defines $g^{\rm red}(\lambda_1,\lambda_2)\colonequals g(\lambda_1,\lambda_2,\frac{1}{\lambda_1\lambda_2})$.
	The zero-pressure case is discussed in Subsection~\ref{smBt}. For the plane-stress normalization considered in Subsection~\ref{smB}, the first two principal Biot stresses are given by the derivatives of the reduced energy
	\begin{align}
		t_1^{\rm red}(\lambda_1,\lambda_2)\colonequals \frac{\partial }{\partial \lambda_1}g^{\rm red}(\lambda_1,\lambda_2)\,,
		\qquad \qquad
		t_2^{\rm red}(\lambda_1,\lambda_2)\colonequals \frac{\partial }{\partial \lambda_2}g^{\rm red}(\lambda_1,\lambda_2)\,.
	\end{align}
	Hence, $t^{\rm red}=(t_1^{\rm red},t_2^{\rm red})$ represents the reduced plane-stress load mapping.
	In order to exclude the so-called Treloar-Kearsley instability \cite{batra2005treloar}, i.e.\ a
	square rubber sheet may turn into a rectangular shape upon equal dead loading on both sides, one may demand strict convexity of $g^{\rm red}$ or, less restrictively, injectivity of
	\begin{align}
		\label{invc}
		(\lambda_1,\lambda_2)\mapsto \matr{t_1^{\rm red}(\lambda_1,\lambda_2) \\t_2^{\rm red}(\lambda_1,\lambda_2)}.
	\end{align}
	A necessary condition for convexity of $g^{\rm red}$ is that its sublevel sets are convex. This is used as a graphical stability\footnote{ADINA$^\text{TM}$ is performing such a test for rubber models.} indicator, e.g.\ in \cite{anssari2023large,anssari2022modelling}.
	While injectivity of \eqref{invc} is well-defined also in the context of incompressibility, convexity of $g^{\rm red}$ in
	terms of $(\lambda_1,\lambda_2)$ does not have a natural intrinsic meaning. This will be clarified below.
	In terms of the principal stretches, we use $g^{\rm red}(\lambda_1,\lambda_2)$ as above and consider
	\begin{align}
		\gtilde_{\rm red}(\log \lambda_1,\log \lambda_2)\colonequals g^{\rm red}(e^{\log \lambda_1},e^{\log \lambda_2})=g(e^{\log \lambda_1},e^{\log \lambda_2},e^{-(\log \lambda_1+\log \lambda_2)})
	\end{align}
	and the monotonicity conditions
	for
	the vector defined in \eqref{invc}, now in terms of the logarithmic stretches. Under the assumption that the third extrinsic Biot principal stress vanishes, one must check whether\footnote{The derivatives do not commute, as shown by the following simple one-dimensional counterexample for $g(\lambda)=\frac12\lambda^2-3\lambda, \ \lambda>0.$
		Let $z=\log\lambda, \ \ghat(z)=g(e^z).$ Then $\dd z\left(\ddd{g}{\lambda}\right) = \lambda g''(\lambda).$ Since $g'(\lambda)=\lambda-3, \ g''(\lambda)=1,$ we obtain $\dd z\left(\ddd{g}{\lambda}\right) = \lambda>0 \ \forall \lambda>0.$ However, $\dd\lambda\left(\ddd{\ghat}{z}\right) = \dd\lambda\left(\lambda g'(\lambda)\right) = g'(\lambda)+\lambda g''(\lambda).$ Therefore $\dd\lambda\left(\ddd{\ghat}{z}\right) = (\lambda-3)+\lambda = 2\lambda-3.$ At $\lambda=1$, this gives $2\lambda-3=-1<0.$ Thus $\dd z\left(\ddd{g}{\lambda}\right)>0$ but $\dd\lambda\left(\ddd{\ghat}{z}\right)<0 \ \text{at } \lambda=1.$ Consequently, the implication $\ddd{g}{z}\left(\ddd{g}{\lambda}\right)>0 \ \to \dd{\lambda}\left(\ddd{\ghat}{z}\right)>0$ is false in general.}
	\begin{align}
		\label{110}
		\sym \D_{(\log \lambda_1,\log \lambda_2)} & [\D_{(\lambda_1,\lambda_2)} g^{\rm red}( \lambda_1, \lambda_2)] {=\sym \. [(\D^2_{(\lambda_1,\lambda_2)} g^{\rm red}( \lambda_1, \lambda_2))\diag (\lambda_1,\lambda_2)]}\in \Symp(2)\,.
	\end{align}
	Since $\R^2$ is convex, condition \eqref{110} implies that $x\mapsto \D_{(\lambda_1,\lambda_2)}g^{\rm red}(e^{x_1},e^{x_2})$ is strictly monotone and hence injective. As $(x_1,x_2)\mapsto(e^{x_1},e^{x_2})$ is bijective, the reduced load mapping \eqref{invc} is injective as well.
	Finally, the related (but non-equivalent) convexity condition on
	\begin{align}
		\label{111}
		(\log \lambda_1,\log \lambda_2)\mapsto \gtilde_{\rm red}(\log \lambda_1,\log \lambda_2)\,,
	\end{align}
	i.e.\ for twice differentiable energies $\D^2_{(\log\lambda_1,\log \lambda_2)}\gtilde_{\rm red}(\log \lambda_1,\log \lambda_2)\in \Sym^{+}(2)$, is equivalent to Hill's inequality for incompressible response \cite{agn_wollner2026concurrent,agn_ghiba2026polyconvexity,baaser2026hyperelastic}.
	The latter condition is related to Hill's inequality
	\cite{agn_voss2026monotonicity,Ogden83,Hill68,hill1968constitutiveb,hill1970,agn_wollner2026concurrent}, while the former is a novel
	requirement that can still serve the purpose of excluding Treloar-Kearsley type instabilities, if desired, see, for example,
	\cite{anssari2023large,anssari2022modelling,batra2005treloar}.

	The paper is now organised as follows. In Section \ref{se1} we define matrix- and vector-monotonicity and recall their equivalence from \cite{agn_voss2026monotonicity}. Then, in Section \ref{se3} we consider matrix- and vector-monotonicity for incompressible materials. After clarifying the correct form of these monotonicity concepts for incompressible materials, we give alternative characterisations and show that they are equivalent. A similar proof of the non-strict spectral monotonicity implication, originating from an earlier version of the present manuscript, is retained in Appendix \ref{app:alternative-spectral-proof}.
	\section{Monotonicity requirements in nonlinear elasticity of compressible elastic materials}\setcounter{equation}{0}\label{se1}
	We consider the general isotropic constitutive equation
	\begin{equation}
		\Sigma_f=\Sigma_f(U)\,, \qquad \Sigma_f\col \Symp(n)\to \Sym(n)\,,
	\end{equation}
	where $\Sigma_f$ denotes a symmetric stress tensor and $U\in\Symp(n)$ is the positive definite stretch tensor. Since direct calculations with tensors (matrices) in three dimensions can be cumbersome, we consider the singular values (the principal stretches) $\lambda_1,\dotsc,\lambda_n$ of $F$, i.e.\ the positive eigenvalues of $U$.
	If $\Sigma_f\col\Symp(n)\to\Symn$ is an isotropic tensor function satisfying
	\begin{align}
		\Sigma_f(Q^T\.\diag(\lambda_1,\dotsc,\lambda_n)\. Q)=Q^T\.\Sigma_f(\diag(\lambda_1,\dotsc,\lambda_n))\. Q\,,
	\end{align}
	then, for all $U=Q^T\diag(\lambda_1,\dotsc,\lambda_n)\.Q$ it holds
	\begin{align}
		\Sigma_f(U)\colonequals
		\Sigma_f(\underbrace{Q^T\diag(\lambda_1,\dotsc,\lambda_n)\. Q}_{U\.\in\Symp(n)})
		=
		\underbrace{Q^T\diag(f(\lambda_1,\dotsc,\lambda_n))\. Q}_{\Sigma_f(U)\.\in\Sym(n)}
		\label{eq:introductionMatrixFunction}
	\end{align}
	with a vector-function $f=(f_1,\dotsc,f_n)\col\Rp^n\to\R^n$ which satisfies
	\begin{align}
		f_i(\lambda_{\pi(1)},\lambda_{\pi(2)},\dotsc,\lambda_{\pi(n)}) = f_{\pi(i)}(\lambda_1,\dotsc,\lambda_n)
	\end{align}
	for any permutation $\pi\col\{1,\dotsc,n\}\to\{1,\dotsc,n\}$. Here, $\Symn$ denotes the space of symmetric $n\times n$ matrices, $\On$ is the orthogonal group and $\diag(\lambda_1,\dotsc,\lambda_n)$ is the diagonal matrix with diagonal entries $\lambda_1,\dotsc,\lambda_n\,.$ The permutation symmetry is equivalent to isotropy of the spectral representation.
	Indeed, {in the compressible case}, in many situations the stress-strain relations are conveniently characterized by the relation between their corresponding principal values, i.e.\ by the relations between the principal stretches $\lambda_1, \lambda_2, \dotsc,\lambda_n$ and the principal forces (principal Biot-stresses)
	\begin{align}
		\label{BT}t_i= \displaystyle\frac{\partial g(\lambda_1,\dotsc,\lambda_n)}{\partial \lambda_i}\,,\qquad i=1,\dotsc,n\,.
	\end{align}
	In \eqref{BT}, $g\colon\R_+^n\to \R$ is the unique permutation symmetric function of the singular values of $U$ (principal stretches) such that $W(F)=\What(U)=g(\lambda_1,\dotsc,\lambda_n)$, where $\R_+^n=(0,\infty)^n$.
	For instance, for the Biot stress tensor defined by
	\begin{align}
		\TBiot=\D_U \What(U)\,,
	\end{align}
	the following {principal stresses} are used
	\begin{align}
		\tBiot=\matr{t_1, & t_2, & \dotsc & t_n}^T.
	\end{align}
	Going back to the general case of an arbitrary isotropic tensor function, the functions $f$ and $\Sigma_f$ related by eq.~\eqref{eq:introductionMatrixFunction} share a number of properties related to injectivity and monotonicity, as we will see in the next sections.
%
%
%
	\subsection{Matrix monotonicity on $\Symp(n)$ versus vector monotonicity on $\R_+^n$}\label{ips2}
	For our purposes, we now recall some related notions of monotonicity.
	\begin{definition}\cite{agn_martin2015some}
		A tensor function $\Sigma_f\col\Symp(n)\to\Symn\.$ is called \emph{strictly matrix-monotone} if
		\begin{align}
			\iprod{\Sigma_f(U)-\Sigma_f(\Ubar)\,,\.U-\Ubar}_{\R^{n\times n}}>0\qquad\forall\;U\neq\Ubar\in\Symp(n)\,.\label{eq:introductionMatrixMonotonicity}
		\end{align}
		We refer to this inequality as \emph{strict matrix-monotonicity} of the tensor function $\Sigma_f$. A tensor function $\Sigma_f\col\Symp(n)\to\Symn\.$ is called \emph{matrix-monotone} if
		\begin{align}
			\iprod{\Sigma_f(U)-\Sigma_f(\Ubar)\,,\.U-\Ubar}_{\R^{n\times n}}\geq 0\qquad\forall\;U,\.\Ubar\in\Symp(n)\,.\label{eq:introductionMatrixMonotonicity1}
		\end{align}
	\end{definition}
	\begin{definition}\cite{agn_martin2015some}
		A vector function $f\col\R^n_+\to\R^n$ is \emph{strictly vector-monotone} if
		\begin{align}
			\iprod{f(\lambda)-f(\lambdabar)\,,\.\lambda-\lambdabar}_{\R^n}>0\qquad\forall\;\lambda\neq\lambdabar\in\R^n_+\,,
		\end{align}
		and it is \emph{vector-monotone} if
		\begin{align}
			\iprod{f(\lambda)-f(\lambdabar)\,,\.\lambda-\lambdabar}_{\R^n}\geq 0\qquad\forall\;\lambda, \lambdabar\in\R^n_+\,,
		\end{align}
	\end{definition}
	\begin{definition}
		A continuously differentiable tensor function $\Sigma_f\col\Symp(n)\to\Symn\.$ is called \emph{strictly differentially matrix-monotone} if
		\begin{align}
			\iprod{\D\Sigma_f(U).H,H} > 0 \qquad\forall\;{U\in\rm Sym}^{++}(n)\,,\;H\in\Sym(n) {\setminus\{0\}}
			\,.
		\end{align}
	\end{definition}
	\begin{definition}\cite{agn_martin2015some}
		A continuously differentiable vector function $f\col\R^n_+\to\R^n$ is called \emph{strictly differentially vector-monotone} if
		\begin{align}
			\iprod{\D f(\lambda).h,h} > 0 \qquad\forall\;\lambda\in\R^n_+\,,\;h\in\R^n {\setminus\{0\}}
			\,.
		\end{align}
	\end{definition}
	Note that $\D f\.(\lambda_1,\dotsc,\lambda_n)$ need not be symmetric. However, for $t=(t_1,\dotsc,t_n)^T=\D g(\lambda_1,\dotsc,\lambda_n)$, we have
	\begin{align}
		\D t(\lambda_1,\dotsc,\lambda_n)
		=
		\D^2 g(\lambda_1,\dotsc,\lambda_n)
		\colonequals
		\left(
		\frac{\partial^2 g}{\partial\lambda_i\partial\lambda_j}
		\right)_{i,j=1,\dotsc,n}
		\in\Sym(n)\,.
	\end{align}
	In \cite{agn_voss2026monotonicity} we discuss in detail the following result, thereby expanding on Ogden's work \cite[last page in the Appendix]{Ogden83}, following Hill's seminal contributions \cite{Hill68,hill1968constitutiveb,hill1970}. For a permutation-invariant set $M\subset\R^n$, we denote by $\Sym(M)\subset\Sym(n)$ the set of symmetric matrices whose eigenvalues belong to $M$.
	\begin{theorem}
		\label{theorem:mainResult}
		A symmetric function $f\col M\subset\R^n\to\R^n$ is (strictly) vector-monotone on $M$ if and only if $\Sigma_f$ is (strictly) matrix-monotone on $\Sym(M)$.
	\end{theorem}
	We use Theorem \ref{theorem:mainResult} as a quoted result throughout the remainder of the paper. To retain the earlier argument developed for this manuscript, a similar proof of the non-strict implication is collected in Appendix \ref{app:alternative-spectral-proof}.
	\subsection{Two--point versus differential monotonicity on $\Symp(n)$}
	Recall that $\Symp(n)$ is an open convex subset of the linear space $\Sym(n)$.
	\begin{definition}
		Let $\Sigma\col \Symp(n)\to\Sym(n)$ be a continuously
		differentiable tensor field.
		Two--point monotonicity means
		\begin{equation}
			\label{eq:HilbertMono}
			\iprod{\Sigma(U)-\Sigma(\widetilde{U})\,,\.U-\widetilde{U}} \geq 0
			\qquad \forall\;U,\widetilde{U}\in\Symp(n)\,.
		\end{equation}
		Differential monotonicity means
		\begin{equation}
			\label{eq:DiffMono}
			\iprod{\D\Sigma(U).H,\.H} \geq 0
			\qquad \forall\;U\in\Symp(n)\,,\ \forall\;H\in\Sym(n)\,.
		\end{equation}
	\end{definition}
	\begin{lemma}[Differential monotonicity implies two--point monotonicity]
		If \eqref{eq:DiffMono} holds, then \eqref{eq:HilbertMono} holds.
	\end{lemma}
	\begin{proof}
		Fix $U,\widetilde{U}\in\Symp(n)$ and consider the line segment
		\begin{align}
			\gamma(s)\colonequals \widetilde{U}+s\.(U-\widetilde{U})\,,\qquad s\in[0,1]\,.
		\end{align}
		By convexity of $\Symp(n)$, $\gamma(s)\in\Symp(n)$
		for all $s\in[0,1]$.
		By the fundamental theorem of calculus,
		\begin{align}
			\Sigma(U)-\Sigma(\widetilde{U})
			=\int_0^1 \D\Sigma(\gamma(s)).[U-\widetilde{U}]\;\ds\,.
		\end{align}
		Taking the Frobenius inner product with $U-\widetilde{U}$ yields
		\begin{align}
			\iprod{\Sigma(U)-\Sigma(\widetilde{U})\,,\.U-\widetilde{U}}
			=\int_0^1 \iprod{\D\Sigma(\gamma(s)).[U-\widetilde{U}],\.U-\widetilde{U}}\;\ds\,.
		\end{align}
		By assumption \eqref{eq:DiffMono}, the integrand is nonnegative for every $s$,
		hence the integral is nonnegative, which proves \eqref{eq:HilbertMono}.
	\end{proof}
	\begin{lemma}[Two--point monotonicity implies differential monotonicity]
		Assume that $\Sigma$ is $C^1$ and that \eqref{eq:HilbertMono} holds.
		Then \eqref{eq:DiffMono} holds.
	\end{lemma}
	\begin{proof}
		Fix $U\in\Symp(n)$ and $H\in\Sym(n)$.
		For sufficiently small $s\in\R$, the matrix $U+s\.H$ belongs to $\Symp(n)$ since the latter is open.
		Applying \eqref{eq:HilbertMono} to $U+s\.H$ and $U$ gives $\iprod{\Sigma(U+s\.H)-\Sigma(U),s\.H} \geq 0.$ Dividing by $s^2>0$ yields
		\begin{align}
			\iprod{\frac{\Sigma(U+s\.H)-\Sigma(U)}{s},\.H} \geq 0\,.
		\end{align}
		Letting $s\to 0$ and using differentiability of $\Sigma$ gives\footnote{For $n=3$ this is equivalent to $\sym \D\Sigma(U)\in \Sym^{+}_4(6)$} $\iprod{\D\Sigma(U).H,\.H} \geq 0$ which proves \eqref{eq:DiffMono}.
	\end{proof}
	In conclusion, on the open convex set $\Symp(n)$, for continuously differentiable
	tensor fields $\Sigma$, two--point matrix monotonicity is equivalent to
	pointwise positive semidefiniteness of the symmetric part of the derivative.

	\subsection{The monotonicity of the Biot-stress tensor for compressible materials}
	We recall that \cite[page 144]{truesdell65}
	\begin{align}
		\label{difs}
		\TBiot=\D_U\What(U)
	\end{align}
	is the symmetric Biot stress tensor and represents ``the principal forces acting in the reference system".
	\noindent Hence, the following holds true for hyperelasticity {of compressible materials}, assuming sufficient regularity:
	For the matrix--vector equivalences below we use Theorem \ref{theorem:mainResult}; the corresponding differential formulations are then expressed through the principal Biot-stress map.
	\begin{align*}
		U\mapsto \TBiot(U) & \ \ \text{matrix-monotone}          \\\quad &\iff\quad (\lambda_1,\dotsc,\lambda_n)\mapsto \tBiot(\lambda_1,\dotsc,\lambda_n) \ \ \text{vector-monotone}\notag\\\quad &\iff\quad \D\tBiot(\lambda_1,\dotsc,\lambda_n)\in \Sym^{+}(n)\quad\forall\;(\lambda_1,\dotsc,\lambda_n)\in\R_+^n\,,
		\\
		U\mapsto \TBiot(U) & \ \ \text{strictly matrix-monotone} \\\quad &\iff\quad (\lambda_1,\dotsc,\lambda_n)\mapsto \tBiot(\lambda_1,\dotsc,\lambda_n) \ \ \text{strictly vector-monotone,}\notag
	\end{align*}
	Note that the monotonicity conditions are global conditions, while the condition $\D\tBiot\in \Symp(n)$ is only local.
	In the present hyperelastic setting, $\D\TBiot(U)=\D^2_U\What(U)$ is major-symmetric due to the existence of the twice differentiable elastic potential $\What$.
	We also note that for $n=3$
	\begin{align}
		U\mapsto \TBiot(U) \ \ \text{strictly differentially matrix-monotone}\quad & \iff\quad \D \TBiot(U)\in \Sym_4^{++}(6)\,,
	\end{align}
	where $\Sym^{++}_4(6)$ denotes the set of positive definite, major-symmetric fourth-order tensors acting on the six-dimensional space $\Sym(3)$.
	%
	%
	%
	%
	%
	\section{Vector monotonicity and matrix monotonicity for incompressible materials }\setcounter{equation}{0}\label{se3}
	For incompressible models in nonlinear elasticity, the determinant constraint leads to a somewhat ambiguous notion of monotonicity. In particular, as we will discuss here, due to the restriction to a submanifold on $\Symp(n)$, differential monotonicity and two-point monotonicity are no longer equivalent if the extrinsic structure of $\Symp(n)$ is considered.
%
%
%
	\subsection{Extrinsic monotonicity on $\SL(n)\cap \Symp(n)$}
	In the following, we will focus on the constraint $\det U=1$ or, equivalently, $\lambda_1\lambda_2\dotsc\lambda_n=1$ for the corresponding principal stretches $\lambda_i$.
	\begin{definition}[Incompressible stretch sets]\label{def:incompressible-stretch-sets}
		We denote by
		\begin{align}
			\mathcal M \colonequals  \{U\in \Symp(n)\. |\det U=1\}\,,\qquad
			\mathcal{S} \colonequals  \{\lambda\in \R^n_+\.|\.\lambda_1\lambda_2\cdots\lambda_n=1\}
		\end{align}
		the manifold of incompressible right stretch tensors and its spectral counterpart, respectively.
	\end{definition}
	For isotropic incompressible models, one typically works with the right stretch
	\begin{align}
		U(F)=\sqrt{F^T F}\in\Symp(n)\,,
		\qquad \det U=\det F=1\,,
	\end{align}
	and defines a tensor field on $\mathcal M$:
	\begin{align}
		\Sigma^{\rm inc}\col \mathcal M\to\Sym(n)\,,
		\qquad \Sigma^{\rm inc}(U)=Q^T\.\diag(f^{\rm inc}(\lambda))\.Q
		\quad\text{for}\quad U=Q^T\diag(\lambda)\.Q\,,\quad \lambda\in 	\mathcal{S}\,.
	\end{align}
	A \emph{naive two--point} analogue of
	matrix monotonicity would be
	\begin{equation}
		\label{eq:twoPointM}
		\iprod{\Sigma^{\rm inc}(U_1)-\Sigma^{\rm inc}(U_2)\,,\.U_1-U_2} \geq 0
		\qquad \forall\;U_1,U_2\in\mathcal M\,,
	\end{equation}
	which is {algebraically} well-defined because the inner product is taken in the ambient linear space $\Sym(n)$.
	Since $\mathcal S$ is permutation invariant and $\Sym(\mathcal S)=\mathcal M$, the naive two--point condition \eqref{eq:twoPointM} is still equivalent to vector monotonicity on $\mathcal S$ by Theorem \ref{theorem:mainResult}.
	However, there are several reasons why matrix monotonicity on $\mathcal{M}$ and vector monotonicity on $\mathcal{S}$ should not be defined in the same way as on $\Symp(n)$, even though the corresponding relations can be interpreted in the ambient space $\Symp(n)$ through the embedding $\mathcal{M}\subset \Symp(n)$.
	Recall that $\SL(n)=\{F\in\R^{n\times n}\.|\det F=1\}$	is a smooth (nonlinear) submanifold of $\R^{n\times n}$, hence \emph{not} a linear subspace and not convex.
	Its tangent space at $F\in \SL(n)$ is
	\begin{equation}
		\label{eq:TangentSL3}
		T_F \SL(n)=\{H\in\R^{n\times n}\.|\ \tr(F^{-1}H)=0\}\,,
	\end{equation}
	since $\dds\det(F+s\.H)\big|_{s=0}=\det F\.\tr(F^{-1}H)=\iprod{{\rm Cof} F, H}$ and $\det F=1$.
	Similarly, the isochoric stretch manifold $\mathcal M\subset\Sym(n)$ is a smooth submanifold of the \emph{linear} space $\Sym(n)$ with tangent space
	\begin{equation}
		\label{eq:TangentM}
		T_U\mathcal M=\{H\in\Sym(n)\.|\ \tr(U^{-1}H)=0\}\,,
	\end{equation}
	since $\dds\det(U+s\.H)\big|_{s=0}=\det U\.\tr(U^{-1}H)=\iprod{\Cof U,H}$.

	On an \emph{open convex} domain $\Omega\subset \R^n$, if $G\col \Omega\to\R^n$ is $C^1$ and satisfies $\iprod{\D G(x).h,h}\geq 0$ for all $x\in\Omega$ and all $h\in\R^n$, then the standard line segment
	argument yields $\iprod{G(x)-G(y)\,,x-y}\geq 0\ \forall\;x,y\in\Omega.$ This works on $\Symp(n)$, too, because it is an open convex subset of the linear space $\Sym(n)$, and the segment $\gamma(s)=\widetilde{U}+s\.(U-\widetilde{U})$ stays in $\Symp(n)$.
	However, $\mathcal M$ and $\SL(n)$ are \emph{not convex} subsets of their ambient vector spaces. In particular,
	for $U,\widetilde{U}\in\mathcal M$, the segment $\widetilde{U}+s\.(U-\widetilde{U})$ typically satisfies $\det(\widetilde{U}+s\.(U-\widetilde{U}))\neq 1$, so it leaves $\mathcal M$.
	Therefore, the fundamental theorem of calculus along the straight segment cannot be used.
	
	\begin{figure}[h!]
		\centering
		\usetikzlibrary{arrows.meta,calc}
		\begin{tikzpicture}[>=Latex, font=\small, scale=1.05]
			\tikzset{
				pt/.style={circle, fill=black, inner sep=1.4pt},
				seg/.style={line width=1.4pt, blue!70!black},
				chord/.style={line width=1.4pt, dashed, red!75!black},
				mani/.style={line width=1.7pt, green!35!black},
				tang/.style={->, line width=1.2pt, black},
				lbl/.style={font=\scriptsize, align=center}
			}
			\begin{scope}[xshift=0cm]
				\node[font=\bfseries] at (0,2.4)
				{convex set $\mathcal E\subset\R^m$};
				\draw[thick, rounded corners=12pt]
				(-2.6,-1.0) rectangle (2.6,1.25);
				\node at (0,0.9) {$\mathcal E$};
				\coordinate (x) at (-1.75,-0.45);
				\coordinate (y) at (1.55,0.35);
				\node[pt,label=below left:$x$] at (x) {};
				\node[pt,label=above right:$y$] at (y) {};
				\draw[seg] (x)--(y);
				\coordinate (xl) at ($(x)!0.27!(y)$);
				\draw[->, line width=1.2pt, blue!70!black]
				(xl) -- ($(xl)!0.0005cm!(y)$);
				\node[blue!70!black] at (0.0,0.50)
				{$x_s=x+s\.(y-x)$};
				\node[lbl,blue!70!black] at (0.55,-0.35)
				{$\dot x_s=y-x$};
				\node[lbl,blue!70!black] at (0,-0.78)
				{straight segment remains in $\mathcal E$};
			\end{scope}
			\draw[densely dashed,gray] (3.35,-1.7)--(3.35,2.55);
			\begin{scope}[xshift=8.8cm]
				\node[font=\bfseries] at (0,2.4)
				{constraint manifold $\mathcal N$};
				\node at (0,1.95)
				{$\mathcal N\in\{\mathcal S,\mathcal M,\SL(n)\}$};
				\coordinate (p) at (-1.1,0.35);
				\coordinate (q) at (1.75,0.35);
				\draw[mani]
				(-2.3,-0.65)
				.. controls (-1.95,0.45) and (-1.40,0.55) ..
				(p)
				.. controls (-0.80,0.15) and (-0.20,-1.00) ..
				(0.55,-0.35)
				.. controls (1.05,0.35) and (1.35,0.85) ..
				(q)
				.. controls (2.05,0.05) and (2.25,-0.45) ..
				(2.45,-0.75);
				\node[pt,label=above left:$p$] at (p) {};
				\node[pt,label=above right:$q$] at (q) {};
				\draw[chord] (p)--(q);
				\node[red!75!black] at (0.35,0.70)
				{$p+s\.(q-p)$};
				\node[lbl,red!75!black] at (0.35,-0.80)
				{straight chord leaves $\mathcal N$};
				\node[green!35!black] at (0.75,1.12)
				{$\gamma(s)\subset\mathcal N$};
				\coordinate (ptan) at (-0.80,0.15);
				\draw[tang]
				(p) -- ($(p)!0.65cm!(ptan)$);
				\node[lbl] at (0.1,0.00)
				{$H\in T_p\mathcal N$};
			\end{scope}
		\end{tikzpicture}
		\caption{Straight segments on convex sets versus admissible curves on constraint manifolds.}
		\label{fig:convex-vs-manifold}
	\end{figure}
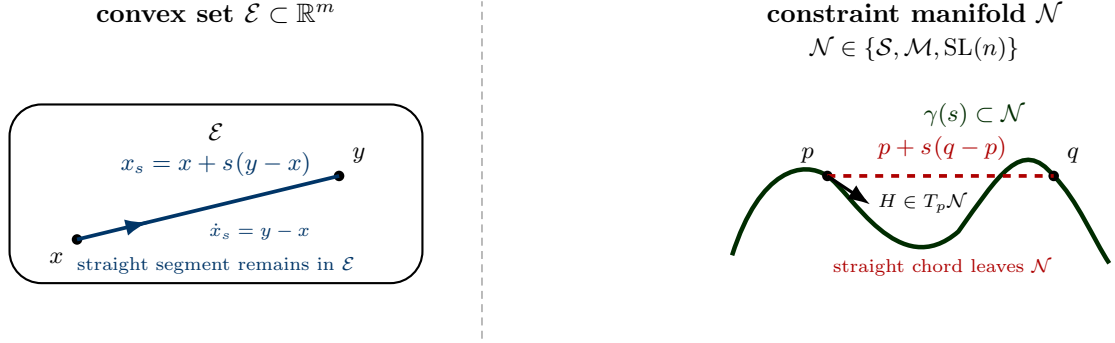
	
	The key difference is that on a nonlinear constraint set, not every tangent direction can be realized by an ambient straight line remaining inside the constraint manifold. Indeed, e.g.\ for $n=3$, let $p=(1,1,1)\in\mathcal S .$
	Since $\mathcal S$ is defined by the constraint $0=q(\lambda_1,\lambda_2,\lambda_3)\colonequals \lambda_1\.\lambda_2\.\lambda_3-1,$ its tangent space at $p$ is $T_p\mathcal S = \{h=(h_1,h_2,h_3)\in\R^n:\D q(p).h=0\}.$ Because $\D q(\lambda_1,\lambda_2,\lambda_3)=(\lambda_2\lambda_3\,,\lambda_1\lambda_3\,,\lambda_1\lambda_2)\,,$ we obtain $\D q(1,1,1)=(1,1,1)\,,$ and therefore $T_p\mathcal S = \{h\in\R^3\.|\.h_1+h_2+h_3=0\}.$ For example, the vector $h=(1,-1,0)\in T_p\mathcal S.$ However, the straight line $p+s\.h=(1+s,\.1-s,\.1)$ does not remain on the manifold, since $(1+s)(1-s)\. 1 = 1-s^2 \neq 1 \quad \text{for } s\neq 0 .$ Thus $p+s\. h\notin\mathcal S$ for $s\neq0$.
	Instead one must follow a curve contained in the manifold, for example $\gamma(s)=\bigl(e^{s},\.e^{-s},\.1\bigr)\,,$ for which $\gamma(s)\in\mathcal S$ since $e^{s}\.e^{-s}\.1=1 .$ Moreover $\gamma(0)=p, \ \gamma'(0)=(1,-1,0)=h .$ Hence the tangent vector $h$ is realized as the velocity of a curve in $\mathcal S$, but not as the direction of a straight chord joining two
	points of the manifold.
	The geometric difference between the two situations is illustrated in Figure~\ref{fig:convex-vs-manifold}.
	
	More precisely, for matrix spaces, fix $U\in\mathcal M$ and $H\in T_U\mathcal M$.
	If one considers the chord $U+s\.H$ in the ambient space $\Sym(n)$, then
	\begin{align}
		\det(U+s\.H)=\det U\.\Big(1+s\tr(U^{-1}H)+O(s^2)\Big)=1+O(s^2)\,,
	\end{align}
	since $\tr(U^{-1}H)=0$ for $H\in T_U\mathcal M$.
	Thus $U+s\.H$ generally leaves $\mathcal M$ at \emph{second order} (unless additional special structure holds).
	Consequently, for a tensor field on $\mathcal M$, $\Sigma^{\rm inc}\col \mathcal M\to\Sym(n)$, a two--point inequality that is only required for pairs of points \emph{on} $\mathcal M$
	does not directly yield the sign of $\iprod{\D\Sigma^{\rm inc}(U).H,H}$ by a difference quotient along $U+s\.H$,
	because $U+s\.H$ is not an admissible comparison point for $s\neq 0$.
	Nevertheless, two--point monotonicity implies the tangential differential condition by differentiating along an admissible curve in $\mathcal M$ with initial velocity $H$.
	In contrast to the case of an open convex subset of $\Sym(n)$, the tangential differential condition is not automatically equivalent (cf.~Fig.~\ref{fig:ambient-vs-tangential} and Example~\ref{example:two_point_counterexample}) to a two--point monotonicity inequality restricted to $\mathcal M$ or to $\SL(n)$. The reason is that arbitrary pairs of points on these manifolds cannot in general be connected by straight segments remaining inside the manifold, and the straight-line difference quotients $U+s\.H$ or $F+s\.H$ are not admissible for $s\neq0$.

	\begin{figure}[h!]
		\centering
		\tikzset{
			monoBox/.style={draw, rounded corners, align=center, inner sep=6pt, font=\small},
			eqv/.style={-Latex, line width=1.9pt},
			impl/.style={-Latex, line width=1.9pt},
			noimpl/.style={-Latex, dashed, line width=1.9pt},
			dbl/.style={<->, Latex-Latex, line width=1.9pt},
			note/.style={font=\scriptsize, align=center}
		}
		\resizebox{0.85\textwidth}{!}{
			\begin{tikzpicture}[>=Stealth, font=\small, node distance=2.8cm]
				\node (naive) [draw, rounded corners, align=center, minimum width=5.0cm]
				{naive two-point on $\mathcal M$ (ambient):\\ $\iprod{\mathcal V (U_1)-\mathcal V (U_2)\,,\.U_1-U_2}\geq0$\\
					\emph{well-defined but not intrinsic}};
				\node (intr) [draw, rounded corners, align=center, right=of naive, minimum width=5.0cm]
				{tangential differential on $\mathcal M$:\\ $\iprod{\D\mathcal V (U).H,\.H}\geq0$\\
					for all $H\in T_U\mathcal M$};
				\draw[->,thick] (naive) -- (intr) node[midway,above] {\tiny };
				\draw[dashed,thick] (intr) -- (naive) node[midway,below] {\tiny no general implication};
				\node (tangent) [below=1.4cm of $(naive)!0.5!(intr)$, draw, rounded corners, align=center, minimum width=10.6cm]
				{reason: $\mathcal M$ is not a linear/convex set. \\ In general $U_1-U_2\notin T_{U_1}\mathcal M$ and $U_2+s\.(U_1-U_2)\notin \mathcal M$.};
				\draw[red,very thick] ($(naive.east)!0.5!(intr.west)+(-0.15,0.20)$) -- ($(naive.east)!0.5!(intr.west)+(0.15,-0.20)$);
				\draw[red,very thick] ($(naive.east)!0.5!(intr.west)+(-0.15,-0.20)$) -- ($(naive.east)!0.5!(intr.west)+(0.15,0.20)$);
		\end{tikzpicture}}
		\caption{Ambient two-point monotonicity versus tangential differential monotonicity on $\mathcal M$.}
		\label{fig:ambient-vs-tangential}
	\end{figure}
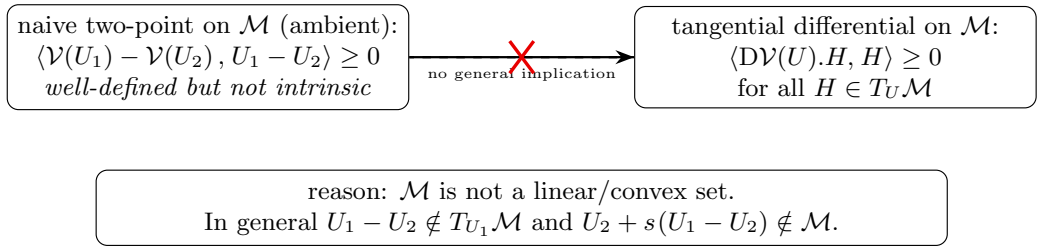
	The relation between
	two--point and differential monotonicity of isotropic tensor fields $\Sigma^{\rm inc}\col \mathcal M\to\Sym(n)$ is discussed further in the next subsections, together with its spectral characterisation.
	%
	%
	\subsection{Intrinsic monotonicity of the Biot principal stresses in spectral variables}
	In order to demonstrate that the above issues are addressed by the notion of monotonicity in terms of logarithmic strain (cf.~eq.~\eqref{logW} and Appendix~\ref{eiA}), we will first consider the Biot stress as a function of the principal stretches. More specifically, we interpret $\lambda_i>0$ as the singular values of $U\in \Symp(n)\cap \SL(n)$. The set $\mathcal S$ is a smooth $(n-1)$-dimensional submanifold of $\R^n_+$ but is
	not convex. Let $f^{\rm inc} \col \mathcal{S} \to \R^n$
	be a vector field $f^{\rm inc} (\lambda) = \bigl(f_1^{\rm inc} (\lambda),f_2^{\rm inc} (\lambda),\dotsc,f_n^{\rm inc} (\lambda)\bigr).$ Recall that the tangent space at $p\in \mathcal{S}$ is $T_p\mathcal S = \{h=(h_1,\dotsc,h_n)\in\R^n\.|\D s(p). h=0\}\,,$ with $s(\lambda)=\lambda_1\lambda_2\dotsc\lambda_n-1$. Its gradient is $\D s(\lambda) = (\lambda_2\lambda_3\dotsc\lambda_n, \lambda_1\lambda_3\dotsc\lambda_n,\dotsc.,\lambda_1\lambda_2\dotsc\lambda_{n-1}).$ Hence, for $h=(h_1,\dotsc,h_n)\in\R^n$ we obtain
	\begin{align}
		\D s(p). h = \lambda_2\lambda_3\dotsc\lambda_n\.h_1 + \lambda_1\lambda_3\dotsc\lambda_n\.h_2 +\dotsc+ \lambda_1\lambda_2\dotsc\lambda_{n-1}\.h_n\,.
	\end{align}
	Since $p\in\mathcal S$, we have $\lambda_1\lambda_2\dotsc\lambda_n=1.$ Dividing the previous relation by $\lambda_1\lambda_2\dotsc\lambda_n$ yields 
	\begin{equation}
		\frac{h_1}{\lambda_1} + \frac{h_2}{\lambda_2}+\dotsc + \frac{h_n}{\lambda_n} =0\,.
	\end{equation}
	Hence,
	\begin{align}
		T_p\mathcal S=
		\{h=(h_1,\dotsc,h_n)\in\R^n\.|\.	\sum_{i=1}^n \frac{h_i}{\lambda_i}=0\}\,.
	\end{align}
		We note again that classical tangential differential monotonicity does not imply two-point monotonicity on the manifold $\mathcal{S}$.
		\begin{example}
			\label{example:two_point_counterexample}
			For $n=2$ and thus $\mathcal S\subset\R^2$, we consider the restriction to $\mathcal S$ of the permutation equivariant function
			\begin{equation}
				f\col\R_+^2\to\R^2\,,
				\qquad f(\lambda_1,\lambda_2)=\matr{\log\lambda_1+\lambda_2\\\log\lambda_2+\lambda_1},
				\qquad
				\D f(\lambda)=\matr{\frac{1}{\lambda_1} & 1\\1 & \frac{1}{\lambda_2}}.
			\end{equation}
			For $p=(\lambda_1,\lambda_2)\in\mathcal{S}$, every tangent vector $h\in T_p\mathcal{S}$ is of the form $h=s\.(\lambda_1,-\lambda_2)$ for some $s\in\R$.
			Thus,
			\begin{align}
				\iprod{Df(p).h,h}&=\iprod{s\matr{\frac{1}{\lambda_1} & 1\\1 & \frac{1}{\lambda_2}}\matr{\lambda_1\\-\lambda_2},\,s\matr{\lambda_1\\-\lambda_2}}
				=s^2\iprod{\matr{1-\lambda_2\\\lambda_1-1},\matr{\lambda_1\\-\lambda_2}}\notag\\
				&=s^2(\lambda_1+\lambda_2-2\.\lambda_1\lambda_2)
				=s^2(\lambda_1+\lambda_2-2)
				\geq 0\,,
			\end{align}
			because $\lambda_1\lambda_2=1$ and $\lambda_1+\lambda_2\geq2$ for all $p=(\lambda_1,\lambda_2)\in\mathcal{S}$. 
			Therefore, $f$ is differentially vector-monotone on $\mathcal{S}$.
			However, $f$ is not two-point vector-monotone on $\mathcal{S}$.
			Indeed, for any $r>0$, we choose $p_1=(e^r,e^{-r})$ and $p_2=(e^{-r},e^r)$ and compute
			\begin{align}
				p_1-p_2&=\matr{e^r-e^{-r}\\e^{-r}-e^r}
				=2\sinh r\matr{1\\-1},\notag\\
				f(p_1)-f(p_2)&=\matr{r+e^{-r}\\-r+e^r}-\matr{-r+e^r\\r+e^{-r}}
				=\matr{2\.r+e^{-r}-e^r\\-2\.r+e^r-e^{-r}}
				=2\.(r-\sinh r)\matr{1\\-1}.
			\end{align}
			Therefore,
			\begin{equation}
				\iprod{f(p_1)-f(p_2),p_1-p_2}=\iprod{2\sinh r\matr{1\\-1},\,2\.(r-\sinh r)\matr{1\\-1}}
				=8\sinh r\.(r-\sinh r)<0\,,
			\end{equation}
			because $\sinh r>r$ for every $r>0$.
			Consequently, tangential differential vector monotonicity on the non-convex manifold $\mathcal{S}$ does not imply two-point vector monotonicity.
		\end{example}
	In order to introduce a more suitable notion of monotonicity on $\mathcal{S}$, we consider the logarithmic mapping
	\begin{align}
		\Phi\col \mathcal S\to \Pi &\colonequals \{u\in\R^n \.|\. u_1+u_2+\dotsc+u_n=0\}\,, \qquad
		{\Phi(\lambda)=(\log\lambda_1,\dotsc,\log\lambda_n)}\,.
	\end{align}
	Since $\log(\lambda_1\lambda_2\dotsc\lambda_n)=0$, we have indeed $\Phi(\lambda)\in\Pi$. Moreover, $\Phi$ is a global diffeomorphism with inverse $u\mapsto (\mathrm e^{u_1},\mathrm e^{u_2},\dotsc,\mathrm e^{u_n})$.
	For $\lambda\in\mathcal S$ and $h\in T_\lambda\mathcal S$, we compute
	\begin{align}
		\D\Phi(\lambda).h = \left(\frac{h_1}{\lambda_1},\frac{h_2}{\lambda_2},\dotsc,\frac{h_n}{\lambda_n}\right)\in \Pi\,.
	\end{align}
	Since $\Pi\subset\R^n$ is a linear vector space, the notion of monotonicity is clearly and unambiguously defined for any function $\ftilde\col\Pi\to\R^n$. Since $\Phi$ is a diffeomorphism, we can also equivalently express any function $f\col\mathcal S\to\R^n$ as $\ftilde=f\circ\Phi\inv\col\Pi\to\R^n$. In the following, we will consider the monotonicity of this function $\ftilde$ on $\Pi$ the \emph{intrinsic} monotonicity of $f$ on $\mathcal S$. More explicitly, we will use the notation
	\begin{align}
		\widetilde f^{\rm inc}(u)
		\colonequals  f^{\rm inc}(e^u)\,,
		\quad u\in\Pi \qquad \iff\qquad \widetilde f^{\rm inc}(\log \lambda)
		=  f^{\rm inc}(\lambda)\,,
		\quad \lambda\in\mathcal S\,.
	\end{align}
	Throughout the following, the term \emph{intrinsic} refers specifically to the structure induced by the logarithmic parametrization, equivalently by the chosen log--Euclidean metric.
	\begin{definition}[Intrinsic vector monotonicity]
	\label{eq:intrinsic_vector_monotonicity}
		The isotropic vector field $f^{\rm inc} \col \mathcal S\to\R^n$ is called \emph{intrinsically vector-monotone} if
		\begin{align}
			\label{im}
			\sum_{i=1}^n
			\bigl(f_i^{\rm inc} (\lambda)-f_i^{\rm inc} (\mu)\bigr)
			\bigl(\log\lambda_i-\log\mu_i\bigr)
			\geq0
			\qquad
			\forall\;\lambda,\mu\in\mathcal S\,.
		\end{align}
	\end{definition}
	Since $\log\lambda-\log\mu\in\Pi$, the monotonicity condition depends only on the projection of the stress difference $f^{\rm inc}(\lambda)-f^{\rm inc}(\mu)$ onto $\Pi$.
	\begin{definition}[Differential intrinsic monotonicity] Assume $f^{\rm inc}$ is continuously differentiable on $\mathcal{S}$. The vector field
		$f^{\rm inc}$ is called \emph{differentially intrinsically monotone} if
		\begin{equation}
			\label{eq:intrinsic-biot}
			\iprod{\D_u\widetilde f^{\rm inc}(u).k,k}_{\R^n} \geq 0
			\quad
			\forall\;u\in\Pi\,\ \forall \;k\in\Pi \quad \iff \quad \iprod{\D_{\log \lambda}\widetilde f^{\rm inc}(\log \lambda).\.k,k}_{\R^n} \geq 0	\quad
			\forall\;\lambda\in\mathcal S\,,\ \forall\;k\in\Pi\,,
		\end{equation}
		where $\Pi=\{u\in\R^n \.|\. u_1+\dotsc+u_n=0\}.$
	\end{definition}
	If the inequality in \eqref{eq:intrinsic-biot} is strict for every $k\in\Pi\setminus\{0\}$, we call $f^{\rm inc}$ strictly differentially intrinsically monotone.
	The above discussion concerning differentials on the manifold $\mathcal S$ becomes considerably simpler after passing to logarithmic coordinates. Indeed, the global diffeomorphism $\Phi\col \mathcal S\to \Pi, \ \Phi(\lambda)=\log\lambda,$ identifies $\mathcal S$ with the linear subspace $\Pi.$ Consequently, after introducing $\widetilde f^{\rm inc}(u) = f^{\rm inc}(e^u)\,, \ u\in\Pi,$ all differentiability properties are understood in the usual Fr\'echet sense on the linear space $\Pi$. In particular, $\D_u\widetilde f^{\rm inc}(u)\col \Pi\to \R^n$ is an ordinary linear operator, so that no distinction between intrinsic and extrinsic differentials is required. The manifold structure of $\mathcal S$ has been completely absorbed into the logarithmic parameterization.
	
	In terms of the differential of $f^{\rm inc}$ on $\mathcal S$, the chain rule yields
		\begin{align}
			\D_u\widetilde f^{\rm inc}(u).k
			=
			\D f^{\rm inc}(\lambda).\bigl[\diag(\lambda)\.k\bigr],
			\qquad
			\lambda=e^u,
		\end{align}
		where $\diag(\lambda)\.k\in T_\lambda\mathcal S$ for every $k\in\Pi$.
		Hence, condition \eqref{eq:intrinsic-biot} can equivalently be written as
		\begin{align}
			\iprod{\D f^{\rm inc}(\lambda).\bigl[\diag(\lambda)\.k\bigr],k}
			\geq0
			\qquad
			\forall\;\lambda\in\mathcal S\,,\quad
			\forall\;k\in\Pi\,.
	\end{align}
	The resulting local--global equivalence is summarized in Figure~\ref{fig:intrinsic-local-global}.
	
	\begin{figure}[h!]
		\centering
		\tikzset{
			monoBox/.style={draw, rounded corners, align=center, inner sep=6pt, font=\small},
			eqv/.style={-Latex, line width=1.9pt},
			impl/.style={-Latex, line width=1.9pt},
			noimpl/.style={-Latex, dashed, line width=1.9pt},
			dbl/.style={<->, Latex-Latex, line width=1.9pt},
			note/.style={font=\scriptsize, align=center}
		}
		\begin{tikzpicture}[>=Stealth, font=\small, node distance=2.4cm]
				\node (PiG) [draw, rounded corners, align=center, minimum width=4.0cm]
				{on $\Pi$ (convex):\\ global monotonicity\\ $\iprod{\widetilde f^{\rm inc}(u)-\widetilde f^{\rm inc}(v)\,,u-v}\geq0$};
				\node (PiL) [draw, rounded corners, align=center, right=of PiG, minimum width=4.0cm]
				{on $\Pi$ (convex):\\ local (differential) monotonicity\\ $\iprod{\D\widetilde f^{\rm inc}(u).k,k}\geq0$};
				\draw[<->,thick,double] (PiG) -- node[above] {\bf equivalence} (PiL);
				\node (bridge) [below=1.5cm of $(PiG)!0.5!(PiL)$, draw, rounded corners, align=center, minimum width=8.5cm]
				{log-identification: $\widetilde f^{\rm inc}(u)=f^{\rm inc}(e^u)$,\quad $u=\log\lambda\in\Pi$\\
					intrinsic condition on $\mathcal S$ becomes Euclidean condition on $\Pi$.};
				\draw[->,thick] (bridge.north) -- (PiG.south);
				\draw[->,thick] (bridge.north) -- (PiL.south);
				\node (Slog) [below=1cm of bridge, draw, rounded corners, align=center, minimum width=4.0cm]
				{on $\mathcal S$ with log-metric:\\
					global: $\sum (f_i(\lambda)-f_i(\mu))(\log\lambda_i-\log\mu_i)\geq0$\\
					local$\iff$global via $\log$};
				\node (Schord) [below=1cm of Slog, draw, rounded corners, align=center, minimum width=4.0cm]
				{on $\mathcal S$ with Euclidean chords (nonconvex):\\
					naive two-point uses $(\lambda-\mu)$\\
					\textbf{no} general local$\iff$global};
				\draw[->,thick] (Slog.north) -- node[right] {\scriptsize } (bridge.south);
				\draw[->,double] (Schord) -- (Slog)
				node[midway,right] {\scriptsize non-equivalent notions} node[midway]{\textcolor{red}{\large\textbf{\textsf{X}}}};
		\end{tikzpicture}
		\caption{Local and global monotonicity in logarithmic coordinates.}
		\label{fig:intrinsic-local-global}
	\end{figure}
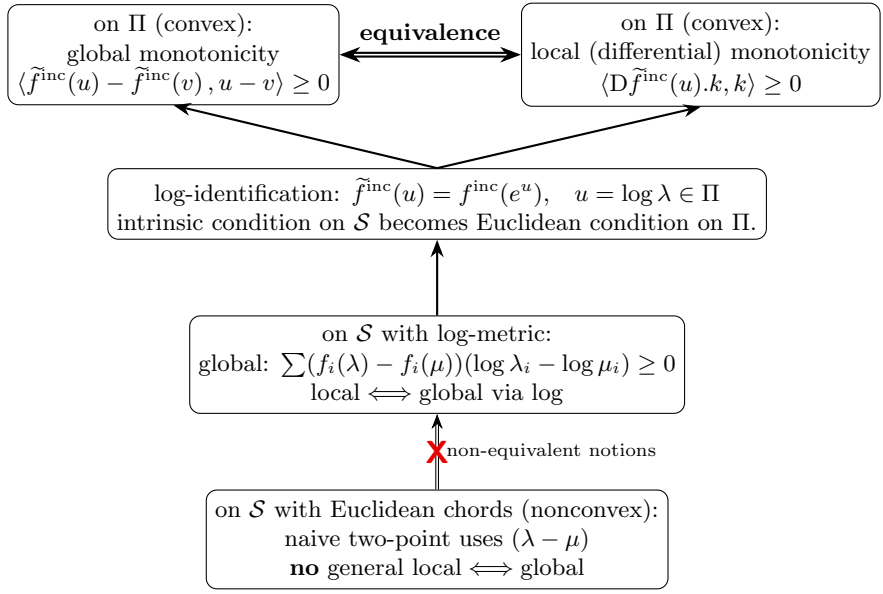
	\begin{theorem}[Local--global equivalence of intrinsic monotonicity]\label{thm:biot-equivalence}
		Let $f^{\rm inc}\in C^1(\mathcal{S})$ be an isotropic vector field.
		Then the following statements are equivalent:
		\begin{enumerate}
			\item $f^{\rm inc}$ is intrinsically monotone in the sense of
			\eqref{im};
			\item $f^{\rm inc}$ is differentially intrinsically monotone in the sense of
			\eqref{eq:intrinsic-biot}
		\end{enumerate}
	\end{theorem}
	\begin{proof}
		In logarithmic coordinates the constraint set $\mathcal S$ is identified with the
		linear subspace $\Pi\subset\R^n$.
		The intrinsic monotonicity condition \eqref{im} becomes
		\begin{align}
			\iprod{\widetilde f^{\rm inc}(u)-\widetilde f^{\rm inc}(v)\,,\.u-v}_{\R^n} \geq 0
			\qquad \forall\;u,v\in\Pi\,.
		\end{align}
		Since $\Pi$ is convex, the standard characterization of monotone operators
		on convex sets applies.
		In particular, the above inequality holds for all $u,v\in\Pi$ if and only if
		\begin{align}
			\hspace{5.5cm} \iprod{\D\widetilde f^{\rm inc}(u).\.k,k}_{\R^n} \geq 0
			\qquad \forall\;u\in\Pi\,,\quad \forall\;k\in\Pi\,.\hspace{6cm} \qedhere
		\end{align}
	\end{proof}
	Moreover, strict positivity in \eqref{eq:intrinsic-biot} for every $k\neq0$ implies strict intrinsic vector monotonicity by integration along line segments in $\Pi$. The converse is not asserted.
	For monotonicity notions based on Euclidean chords $(\lambda-\mu)$, the standard straight-line argument does not provide an analogous local--global equivalence because $\mathcal S$ is not convex.
	
	The identification of $\mathcal S$ with $\Pi$ can also be interpreted in a more direct geometrical way, which further motivates the definition of intrinsic monotonicity via the logarithmic mapping. More specifically,
	we equip $\mathcal S$ with the chosen log--Euclidean metric, defined as the pull-back of the Euclidean metric on $\Pi$ by $\Phi$, namely
	\begin{align}
		\label{def_metric}
		\iprod{h,k}_\mathcal{S} \colonequals \iprod{\D\Phi(\lambda).h,\D\Phi(\lambda).k}_{\R^n}
		= \sum_{i=1}^n \frac{h_i\.k_i}{\lambda_i^2}
		\qquad\forall\; h,k\in T_\lambda\mathcal S\,.
	\end{align}
	By construction, $\Phi$ defines an isometry between $\mathcal S$ with this metric and $\Pi$ with the Euclidean metric, i.e.
	
	$\Phi$ preserves lengths and angles, and hence also distances.
	In particular, $\Phi$ maps minimizing geodesics in $\mathcal S$ to straight segments in $\Pi$ and vice versa (see e.g.\ \cite{DoCarmo}).
	By definition of the pull-back metric, for any smooth curve $\gamma$ in $\mathcal S$ we have
	\begin{align}
		\norm{\gamma'(s)}_{\mathcal S}
		=
		\norm{\D\Phi(\gamma(s)).\.\gamma'(s)}
		=
		\norm{\dds(\Phi\circ\gamma)(s)}.
	\end{align}
	Hence,
	\begin{align}
		\mathrm{Length}_{\mathcal S}(\gamma)
		=
		\mathrm{Length}_{\Pi}(\Phi\circ\gamma)\,.
	\end{align}
	Therefore, a curve $\gamma$ in $\mathcal S$ minimizes the Riemannian distance between $\lambda$ and $\mu$ if and only if $\Phi\circ\gamma$ minimizes the Euclidean distance between $u=\log\lambda$ and $v=\log\mu$ in $\Pi$.
	Since $\Pi$ is a Euclidean space, the unique straight segment joining $u=\Phi(\lambda)=\log\lambda$ and $v=\Phi(\mu)=\log\mu$ is the straight segment
	\begin{align}
		\widetilde\gamma(s)=(1-s)\.u+s\.v, \qquad s\in[0,1]\,.
	\end{align}
	Because $\Phi\col (\mathcal S,g)\to(\Pi,\iprod{\.,\.})$ is an isometry, it follows that the unique affinely parametrized minimizing geodesic in $\mathcal S$ joining $\lambda$ and $\mu$ is obtained by pulling back the segment $\widetilde\gamma$ via $\Phi^{-1}$. 
	Therefore,
	\begin{align}
		\gamma(s)
		=
		\Phi^{-1}(\widetilde\gamma(s))
		=
		\exp\bigl((1-s)\log\lambda+s\log\mu\bigr)\,,
		\qquad s\in[0,1]\,.
	\end{align}
	In particular, there exists a unique affinely parametrized minimizing geodesic in $\mathcal S$
	connecting $\lambda$ and $\mu$.
	Moreover, since $\Phi$ is an isometry, the induced Riemannian distance on $\mathcal S$ coincides with the Euclidean distance in $\Pi$, see \cite{DoCarmo}, i.e.
	\begin{align}
		d(\lambda,\mu)
		=
		\|\log\lambda-\log\mu\|_{\R^n}.
	\end{align}
	The curve $\gamma$ realizes this minimal length. 
	In particular, the two-point monotonicity according to Definition~\ref{eq:intrinsic_vector_monotonicity} is naturally associated with the log--Euclidean geodesic displacement $\log\lambda-\log\mu$, whereas straight line segments (which would correspond to the extrinsic notion of monotonicity) are generally not contained in $\mathcal S$.
	The geometry of the logarithmic identification is summarized in Figure~\ref{fig:log-identification}.
	
	\begin{figure}[h!]
		\centering
		\tikzset{
			monoBox/.style={draw, rounded corners, align=center, inner sep=6pt, font=\small},
			eqv/.style={-Latex, line width=1.9pt},
			impl/.style={-Latex, line width=1.9pt},
			noimpl/.style={-Latex, dashed, line width=1.9pt},
			dbl/.style={<->, Latex-Latex, line width=1.9pt},
			note/.style={font=\scriptsize, align=center}
		}
			\begin{tikzpicture}[>=Stealth, node distance=2.5cm, font=\small]
				\node (S) [draw, rounded corners, align=center, text width=6cm]
				{$\mathcal S=\{\lambda\in\R_+^n\.|\.\lambda_1\lambda_2\dotsc\lambda_n=1\}$\par (nonconvex)};
				\node (Pi) [draw, rounded corners, align=center, text width=6.0cm, right=of S]
				{$\Pi=\{u\in\R^n\.|\. u_1+u_2+\dotsc+u_n=0\}$\par (linear, convex)};
				\draw[<->,thick] (S) -- node[above] {$\Phi(\lambda)=\log\lambda$} (Pi);
				\path[] (Pi) -- node[below] {$\Phi^{-1}(u)=\exp u$} node[midway](center){} (S);
				\node (metric) [below=1.75cm of center, anchor=north, draw, rounded corners, align=center]
				{metric on $\mathcal S$ defined by pull-back:\\ $\iprod{h,k}_\mathcal{S} \colonequals \iprod{\D\Phi(\lambda).h,\D\Phi(\lambda).\.k}$\\ $\Rightarrow \Phi$ is an isometry, and $(\mathcal S,\text{log-metric})\cong (\Pi,\text{Euclidean})$.};
				\draw[->,thick] (S) -- (metric);
				\draw[->,thick] (Pi) -- (metric);
		\end{tikzpicture}
		\caption{Logarithmic identification of $\mathcal S$ with the linear space $\Pi$.}
		\label{fig:log-identification}
	\end{figure}
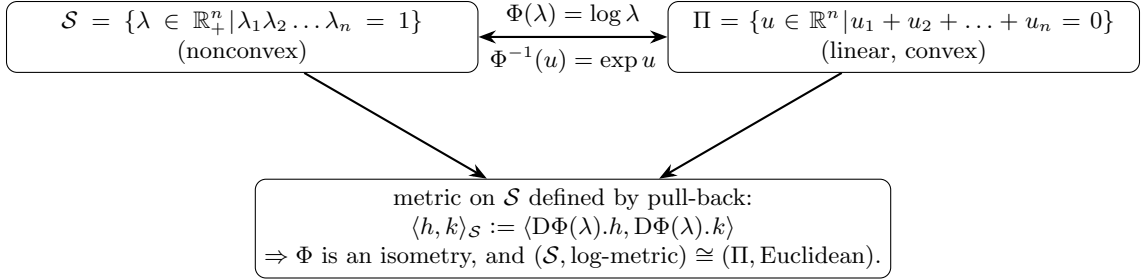
	
	The intrinsic monotonicity of $f^{\rm inc}$ is the natural spectral counterpart of matrix monotonicity for isotropic tensor fields on $\mathcal M$, which we will consider next.
	\subsection{Intrinsic matrix monotonicity}
	Recall that $\mathcal M$ is the manifold of isochoric right stretch tensors. This manifold is naturally
	associated with $\SL(n)$: for $F\in \SL(n)$ with polar decomposition $F=R\.U$ (see \cite{agn_neff2014grioli})\,, one has $U\in\mathcal M$.
	Introduce the matrix logarithm
	\begin{align}
		X \colonequals  \log U\,.
	\end{align}
	Since $\det U=1$, it follows that $\tr(\log U)=\log(\det U)=0,$ hence
	\begin{align}
		X \in \dev\Sym(n)
		\colonequals  \{ X\in \Sym(n) \.|\tr(X)=0 \}\,.
	\end{align}
	The map
	\begin{align}
		\Phi\col \mathcal M \to \dev\Sym(n)\,,
		\qquad
		\Phi(U)=\log U
	\end{align}
	is a global diffeomorphism with inverse $X\mapsto \exp X$. Thus, $\mathcal M$
	can be identified with the linear space $\dev\Sym(n)$ via the logarithmic map.
	\medskip\par
	We equip $\mathcal M$ with the chosen \emph{log--Euclidean metric}, i.e.\ the pull-back of the Frobenius inner product $\iprod{A,B} \colonequals \tr(A^T B)=\tr(AB)$ from $\dev\Sym(n)$ via $\Phi$. In particular, for $U\in\mathcal M$ and $H,K\in T_U\mathcal M$, we set
	\begin{align}
		\iprod{H,K}_U
		\colonequals
		\iprod{\D\Phi(U).\.H,\; \D\Phi(U).\.K}\,.
	\end{align}
	The differential $\D\Phi(U)$ defines an isometric isomorphism
	\begin{align}
		\D\Phi(U)\col T_U\mathcal M \to \dev\Sym(n)\,.
	\end{align}
	with respect to this metric, thus geodesics are given by
		\begin{align}
			\gamma(s)=\exp\bigl((1-s)\.X + s\.Y\bigr),
			\qquad X=\log U,\quad Y=\log \widetilde{U}\,,
		\end{align}
	and the Riemannian distance reads
		\begin{align}
			d(U,\widetilde{U})=\norm{\log U-\log \widetilde{U}}\,.
		\end{align}
	We will therefore consider a definition analogous to $\mathcal S$ for the intrinsic monotonicity on $\mathcal M$.
	Let
	\begin{align}
		\Sigma^{\rm inc}\col \mathcal M \to \Sym(n)
	\end{align}
	be a (for instance isotropic) tensor field.
	\begin{definition}[Intrinsic log-monotonicity]
		The field $\Sigma^{\rm inc}$ is called intrinsically monotone on $\mathcal M$ if
		\begin{equation}
			\label{eq:matrix_intrinsic_global}
			\iprod{\Sigma^{\rm inc}(U)-\Sigma^{\rm inc}(\widetilde{U})\,,\ \log U - \log \widetilde{U}}
			\geq 0
			\qquad \forall\;U,\widetilde{U}\in\mathcal M\,.
		\end{equation}
	\end{definition}
	Define the log-reparametrized mapping
	\begin{align}
		\widetilde\Sigma^{\rm inc}\col
		\dev\Sym(n)\to \Sym(n)\,,
		\qquad
		\widetilde\Sigma^{\rm inc}(X) \colonequals  \Sigma^{\rm inc}(\exp X)\,.
	\end{align}
	\begin{definition}[Differential intrinsic monotonicity]
		$\Sigma^{\rm inc}$ is called differentially intrinsically monotone if
		\begin{equation}
			\label{eq:matrix_intrinsic_local}
			\iprod{\D\widetilde\Sigma^{\rm inc}(X).H,\.H}
			\geq 0
			\qquad
			\forall\;X\in \dev\Sym(n)\,,
			\quad H\in \dev\Sym(n)\,.
		\end{equation}
	\end{definition}
	If the inequality in \eqref{eq:matrix_intrinsic_local} is strict for every $H\in\dev\Sym(n)\setminus\{0\}$, we call $\Sigma^{\rm inc}$ strictly differentially intrinsically monotone.
	The logarithmic parametrization removes the geometric difficulties associated with differentiating tensor fields on the nonlinear manifold $\mathcal M$.
	Indeed, $\log\col\mathcal M \to \dev\Sym(n)$ is a global diffeomorphism, and $\dev\Sym(n)$ is a linear vector space.
	Consequently, after introducing the log-reparametrized mapping $\widetilde\Sigma^{\rm inc}(X) \colonequals \Sigma^{\rm inc}(\exp X)\,,$ all derivatives are understood in the ordinary Fr\'echet sense on the linear space $ \dev\Sym(n)$.
	In particular, $\D\widetilde\Sigma^{\rm inc}(X)\col \dev\Sym(n) \to \Sym(n)$ is an ordinary linear operator, and no distinction between intrinsic and extrinsic differentials is required. Therefore, the differential intrinsic monotonicity condition is naturally formulated in logarithmic coordinates.
	\begin{theorem}[Intrinsic equivalence]\label{thm:matrix_equivalence}
		Let $\Sigma^{\rm inc}\in C^1(\mathcal M)$.
		Then the following are equivalent:
		\begin{enumerate}
			\item $\Sigma^{\rm inc}$ is intrinsically monotone in the sense of
			\eqref{eq:matrix_intrinsic_global},
			\item $\Sigma^{\rm inc}$ is differentially intrinsically monotone in the sense of
			\eqref{eq:matrix_intrinsic_local}.
		\end{enumerate}
	\end{theorem}
	\begin{proof}
		Via the diffeomorphism $U=\exp X$, condition
		\eqref{eq:matrix_intrinsic_global} becomes
		\begin{align}
			\iprod{\widetilde\Sigma^{\rm inc}(X)-\widetilde\Sigma^{\rm inc}(Y)\,,\ X-Y}
			\geq 0
			\qquad \forall\;X,Y\in \dev\Sym(n)\,.
		\end{align}
		Since $\dev\Sym(n)$ is a linear convex space,
		the classical characterization of monotone operators applies:
		the above holds if and only if
		\begin{align}
			\iprod{\D\widetilde\Sigma^{\rm inc}(X).H,H}\geq 0
			\qquad
			\forall\;X\in \dev\Sym(n)\,,
			\quad\forall\;H\in \dev\Sym(n)\,.
		\end{align}
	\end{proof}
	As in the vector case, strict positivity in \eqref{eq:matrix_intrinsic_local} for every nonzero $H$ implies strict intrinsic matrix monotonicity.
	%
	%
	\subsection{Vector monotonicity versus matrix monotonicity for incompressible materials}
	Let
	\begin{equation}
		\Sigma^{\rm inc}_f\col \mathcal M \to \Sym(n)\,, \qquad \Sigma^{\rm inc}_f=\Sigma^{\rm inc}_f(U)
	\end{equation}
	be a general isotropic constitutive equation for incompressible elastic materials,
	where $\Sigma^{\rm inc}$ denotes a symmetric stress tensor and $U\in\Symp(n)\cap \SL(n)$ is the positive definite stretch tensor. Since $\Sigma_f^{\rm inc}\col\mathcal M \to \Sym(n)$ is an isotropic tensor function, it satisfies
	\begin{align}
		\Sigma_f^{\rm inc}(Q^T\.\diag(\lambda_1,\dotsc,\lambda_n)\. Q)=Q^T\.\Sigma_f^{\rm inc}(\diag(\lambda_1,\dotsc,\lambda_n))\. Q
	\end{align}
	and therefore
	\begin{align}
		\Sigma_f^{\rm inc}(U)\colonequals \Sigma_f^{\rm inc}(\underbrace{Q^T\.\diag(\lambda_1,\dotsc,\lambda_n)\. Q}_{U\in\mathcal M}) = \underbrace{Q^T\.\diag(f^{\rm inc}(\lambda_1,\dotsc,\lambda_n))\. Q}_{\Sigma_f^{\rm inc}(U)\.\in\Sym(n)}
		\label{eq:introductionMatrixFunction1}
	\end{align}
	with a vector-function $f^{\rm inc}=(f_1^{\rm inc},f_2^{\rm inc},\dotsc,f_n^{\rm inc})\col\mathcal S \to\R^n$ which satisfies
	\begin{align}
		f_i^{\rm inc}(\lambda_{\pi(1)},\lambda_{\pi(2)},\dotsc,\lambda_{\pi(n)}) = f_{\pi(i)}^{\rm inc}(\lambda_1,\dotsc,\lambda_n)
	\end{align}
	for any permutation $\pi\col\{1,\dotsc,n\}\to\{1,\dotsc,n\}$.
	{Of course, for $U=Q^T\.\diag(\lambda)\.Q$,
		we have $\log U = Q^T\.\diag(\log\lambda)\.Q$. Through the mappings
		$\widetilde\Sigma^{\rm inc}\col \dev\Sym(n)\to \Sym(n)$, $\widetilde\Sigma^{\rm inc}(X) \colonequals \Sigma^{\rm inc}(\exp X)$ and $\widetilde f^{\rm inc}\col \Pi\to \R^n$, $\widetilde f^{\rm inc}(u) \colonequals f^{\rm inc}(e^u)$ we may write the isotropy condition as
		\begin{align}
			\widetilde \Sigma_f^{\rm inc}(X)\colonequals \widetilde \Sigma_f^{\rm inc}(\underbrace{Q^T\.\diag(u_1,\dotsc,u_n)\. Q}_{X\in\dev\Sym(n)}) = \underbrace{Q^T\.\diag(\widetilde f^{\rm inc}(u_1,\dotsc,u_n))\. Q}_{\widetilde\Sigma_f^{\rm inc}(X)\.\in\Sym(n)} \qquad\forall\;Q\in\On\label{eq:introductionMatrixFunction2}
		\end{align}
		where $X=\log U\in \dev\Sym(n)$ and $u=(u_1,\dotsc,u_n)=(\log \lambda_1,\dotsc,\log \lambda_n)\in \Pi$.
	}
	Since $\Pi$ is permutation-invariant and $\Sym(\Pi)=\dev\Sym(n)$, the spectral equivalence after logarithmic reparametrization is a direct application of Theorem \ref{theorem:mainResult}.
	\begin{proposition}[Spectral equivalence]\label{prop:spectral-equivalence}
		Intrinsic matrix monotonicity, respectively strict intrinsic matrix monotonicity,
		of tensor fields $\Sigma_f^{\rm inc}$ on $\mathcal M$ is equivalent to
		intrinsic vector monotonicity, respectively strict intrinsic vector monotonicity,
		of $f^{\rm inc}$ on $\mathcal S$, cf.\ Figure~\ref{fig:intrinsic-spectral-equivalence}.
	\end{proposition}
	
	%
	%
	
	\begin{figure}[h!]
		\centering
		\tikzset{
			monoBox/.style={draw, rounded corners, align=center, inner sep=6pt, font=\small},
			eqv/.style={-Latex, line width=1.9pt},
			impl/.style={-Latex, line width=1.9pt},
			noimpl/.style={-Latex, dashed, line width=1.9pt},
			dbl/.style={<->, Latex-Latex, line width=1.9pt},
			note/.style={font=\scriptsize, align=center}
		}
			\begin{tikzpicture}[>=Stealth, font=\small, node distance=2.6cm]
				\node (M) [draw, rounded corners, align=center, minimum width=5.2cm]
				{$\mathcal M=\{U\in\Symp(n)\.|\det U=1\}$\\
					intrinsic matrix monotonicity\\ $\iprod{\Sigma^{\rm inc}(U)-\Sigma^{\rm inc}(\widetilde{U})\,,\log U-\log \widetilde{U}}\geq0$};
				\node (S) [draw, rounded corners, align=center, right=of M, minimum width=5.2cm]
				{$\mathcal S=\{\lambda\in\R_+^n\.|\.\lambda_1\lambda_2\dotsc\lambda_n=1\}$\\
					intrinsic vector monotonicity\\ $\iprod{f^{\rm inc}(\lambda)-f^{\rm inc}(\mu)\,,\log\lambda-\log\mu}\geq0$};
				\draw[<->,very thick] (M) -- node[above,align=center]
				{isotropy } (S);
		\end{tikzpicture}
		\caption{Spectral equivalence of intrinsic vector and matrix monotonicity also holding in the strict case.}
		\label{fig:intrinsic-spectral-equivalence}
	\end{figure}
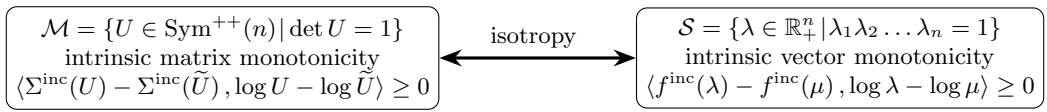
	
	\begin{proposition}[Differential spectral equivalence]
		\label{prop:differential-spectral-equivalence}
		Let $\widetilde f^{\rm inc}\in C^1(\Pi)$ be permutation-equivariant and let $\widetilde\Sigma_f^{\rm inc}\col\dev\Sym(n)\to\Sym(n)$ be the corresponding isotropic tensor function defined by \eqref{eq:introductionMatrixFunction2}.
		Then
		\begin{align}
			\iprod{\D\widetilde f^{\rm inc}(u).k,k}\geq0
			\qquad
			\forall\;u\in\Pi\,,\quad k\in\Pi
		\end{align}
		if and only if
		\begin{align}
			\iprod{\D\widetilde\Sigma_f^{\rm inc}(X).H,H}\geq0
			\qquad
			\forall\;X\in\dev\Sym(n)\,,\quad H\in\dev\Sym(n)\,.
		\end{align}
		The same equivalence holds with strict inequalities for all non-zero directions $k$ and $H$.
	\end{proposition}
	\begin{proof}
		Due to isotropy, it is sufficient to consider $X=\diag(u)$ with $u\in\Pi$.
		For arbitrary $H\in\dev\Sym(n)$, let $h=(H_{11},\dotsc,H_{nn})\in\Pi$.
		All derivatives of $\widetilde f^{\rm inc}$ in the following are understood as derivatives of the map $\widetilde f^{\rm inc}\col\Pi\to\R^n$ on the linear space $\Pi$.
		The derivative of an isotropic spectral tensor function separates into the contribution due to changes of the eigenvalues and the contribution due to changes of the corresponding eigendirections, cf.\ \cite{carlson1986derivative,chadwick1971theorem}. More precisely,
		\begin{align}
			\iprod{\D\widetilde\Sigma_f^{\rm inc}(X).H,H}
			&=
			\iprod{\D\widetilde f^{\rm inc}(u).h,h}
			+
			\sum_{i\neq j}
			\frac{\widetilde f_i^{\rm inc}(u)-\widetilde f_j^{\rm inc}(u)}{u_i-u_j}\,H_{ij}^2\,,
			\label{eq:differential-spectral-formula}
		\end{align}
		where, if $u_i=u_j$, the quotient is understood by continuous extension\footnote{
			To see this, fix $v\in\Pi$ with $v_i=v_j$, let $d=e_i-e_j$, and consider the one-parameter curve $u(s)=v+s\.d$, $s\in\mathbb R$.
			Permutation equivariance implies $\widetilde f_i^{\rm inc}(v)=\widetilde f_j^{\rm inc}(v)$.
			Thus, with $q(u)\colonequals\iprod{\widetilde f^{\rm inc}(u),d} =\widetilde f_i^{\rm inc}(u)-\widetilde f_j^{\rm inc}(u)$,
			\begin{align*}
				\frac{\widetilde f_i^{\rm inc}(v+s\.d)-\widetilde f_j^{\rm inc}(v+s\.d)}{(v_i+s)-(v_j-s)}
				=
				\frac{q(v+s\.d)-q(v)}{2\.s}
				\xrightarrow{s\to0}
				\frac12\.\iprod{\D q(v),d}
				=
				\frac12\.\iprod{\D\widetilde f^{\rm inc}(v).d,d}\,.
			\end{align*}}
		with value
		\begin{align}
			\frac{1}{2}\.\iprod{\D\widetilde f^{\rm inc}(u).(e_i-e_j),e_i-e_j}\,.
			\label{eq:continuous-divided-difference}
		\end{align}
		Assume first that $\widetilde f^{\rm inc}$ is differentially vector-monotone on $\Pi$.
		Since $\Pi$ is convex, this implies vector monotonicity on $\Pi$.
		For $u_i\neq u_j$, let $\mathcal{P}^{(i,j)}u$ denote the vector obtained by interchanging the components $u_i$ and $u_j$.
		Permutation equivariance gives $\widetilde f^{\rm inc}(\mathcal{P}^{(i,j)}u)=\mathcal{P}^{(i,j)}\widetilde f^{\rm inc}(u)$, and therefore
		\begin{align}
			0\leq
			\iprod{\widetilde f^{\rm inc}(u)-\widetilde f^{\rm inc}(\mathcal{P}^{(i,j)}u), u-\mathcal{P}^{(i,j)}u}
			&=
			2\.\bigl(\widetilde f_i^{\rm inc}(u)-\widetilde f_j^{\rm inc}(u)\bigr)(u_i-u_j)\notag\\
			\iff\qquad
			\frac{\widetilde f_i^{\rm inc}(u)-\widetilde f_j^{\rm inc}(u)}{u_i-u_j}&\geq0\,.
		\end{align}
		If $u_i=u_j$, the same conclusion follows from \eqref{eq:continuous-divided-difference}, since $e_i-e_j\in\Pi$.
		Thus both terms on the right-hand side of \eqref{eq:differential-spectral-formula} are non-negative, which proves differential matrix monotonicity.
		
		Conversely, choose $X=\diag(u)$ and $H=\diag(k)$ with $k\in\Pi$.
		Then all off-diagonal entries of $H$ vanish, and
		\begin{equation}
			\iprod{\D\widetilde\Sigma_f^{\rm inc}(X).H,H}
			=
			\iprod{\D\widetilde f^{\rm inc}(u).k,k}\,.
		\end{equation}
		Hence differential matrix monotonicity implies differential vector monotonicity.
		The strict case follows analogously.
	\end{proof}
	In particular, the preceding global and differential spectral equivalences apply to the extension-independent representative $\TBiot^{\rm inc}$ introduced in the introduction.
	
	%
	%
	%
	\subsection{Monotonicity of the zero-pressure extrinsic Biot-stress tensor for incompressible materials}\label{smBt}
	Consider the extension-dependent zero-pressure extrinsic Biot-stress tensor for incompressible materials, i.e.\ $\TBiot^{\rm ext,0}\col \mathcal M\to \Sym(n)$, defined by
	\begin{align}
		\TBiot^{\rm ext,0}(U)\colonequals \D_U\What(U)\in\Sym(n)\,.
	\end{align}
	In this case, due to the isotropy of this tensor field, the corresponding principal Biot stresses are
	\begin{align}
		(\tBiot^{\rm ext,0})_i
		=
		\frac{\partial g}{\partial \lambda_i}(\lambda_1,\dotsc,\lambda_n)\,,
		\qquad i=1,\dotsc,n.
	\end{align}
	Define the log-reparametrized mappings
	\begin{align}
		\TBiotTilde^{\rm ext,0}\col
		\dev\Sym(n)\to \Sym(n)\,,
		\qquad
		\TBiotTilde^{\rm ext,0}(X) \colonequals \TBiot^{\rm ext,0}(\exp X)
	\end{align}
	and
	\begin{align}
		\tBiotTilde^{\rm ext,0}\col \Pi \to \R^n,
		\qquad
		\tBiotTilde^{\rm ext,0}(u) \colonequals \tBiot^{\rm ext,0}(e^u)\,.
	\end{align}
	\noindent Hence, the following holds true for hyperelasticity {of incompressible materials}, assuming sufficient regularity:
	\begin{align*}
		\TBiotTilde^{\rm ext,0} & \ \ \text{matrix-monotone on }\dev\Sym(n)          \\\quad &\iff\quad \tBiotTilde^{\rm ext,0}\ \ \text{vector-monotone on }\Pi\notag\\\quad &\iff\quad \iprod{\D_u\tBiotTilde^{\rm ext,0}(u).\,k,\;k}_{\R^n} \geq 0
		\quad \forall\;u\in\Pi\,,\ \forall\;k\in\Pi\,,
		\\
		\TBiotTilde^{\rm ext,0} & \ \ \text{strictly matrix-monotone on }\dev\Sym(n) \\\quad &\iff\quad \tBiotTilde^{\rm ext,0}\ \ \text{strictly vector-monotone on }\Pi\,.\notag
	\end{align*}
	Moreover, the pointwise condition
	\begin{align}
		\iprod{\D_{\log \lambda}\tBiotTilde^{\rm ext,0}(\log \lambda).\.k,\;k}_{\R^n} > 0
		\qquad \forall\;\lambda\in\mathcal S\,,\ \forall\;k\in\Pi\setminus\{0\}
		\label{eq:strong-vector-intrinsic-biot}
	\end{align}
	is sufficient for strict intrinsic vector monotonicity and therefore, by Proposition~\ref{prop:spectral-equivalence}, for strict intrinsic matrix monotonicity.
	By Proposition~\ref{prop:differential-spectral-equivalence}, condition~\eqref{eq:strong-vector-intrinsic-biot} is moreover equivalent to strict differential intrinsic matrix monotonicity of $\TBiotTilde^{\rm ext,0}$.
	Note that the zero-pressure extrinsic Biot-stress tensor for incompressible materials is the restriction of the standard ambient
	Biot stress, i.e.\ for compressible materials $\TBiot(U)=\D_U\What(U)\,, \ U\in \Symp(3)\,,$ restricted to the incompressibility manifold $\mathcal M.$ Thus, in this subsection no pressure term is subtracted; equivalently, we select the particular extrinsic representative corresponding to the pressure choice $p=0$.
	Let the elastic ambient energy be isotropic, so that $\What(U)=g(\lambda_1,\lambda_2,\lambda_3)\,,$ where $\lambda_i>0$ are the principal stretches of $U$. If $U=Q^T\.\diag(\lambda_1,\lambda_2,\lambda_3)\.Q,$ then, by isotropy, $\D_U\What(U) = Q^T\.\diag(\tBiot^{\rm ext,0})\.Q,$ where the principal Biot stresses are $(\tBiot^{\rm ext,0})_i=\frac{\partial g}{\partial \lambda_i}(\lambda_1,\lambda_2,\lambda_3)$, $i=1,2,3.$ Due to the incompressibility constraint $\lambda_1\lambda_2\lambda_3=1,$ we write $\lambda_3=\frac{1}{\lambda_1\lambda_2}$ and introduce the reduced energy
	\begin{align}
		g^{\rm red}\col \R^2_+\to \R,\qquad g^{\rm red}(\lambda_1,\lambda_2)
		\colonequals
		g\left(\lambda_1,\lambda_2,
		\frac{1}{\lambda_1\lambda_2}\right).
	\end{align}
	By the chain rule, $\frac{\partial \lambda_3}{\partial \lambda_1} = -\frac{\lambda_3}{\lambda_1}, \ \frac{\partial \lambda_3}{\partial \lambda_2} = -\frac{\lambda_3}{\lambda_2}.$ Therefore,
	\begin{align}
		\frac{\partial g^{\rm red}}{\partial \lambda_1}
		& =
		\frac{\partial g}{\partial \lambda_1} -
		\frac{\lambda_3}{\lambda_1}
		\frac{\partial g} {\partial \lambda_3}
		=
		(\tBiot^{\rm ext,0})_1-\frac{\lambda_3}{\lambda_1} (\tBiot^{\rm ext,0})_3\,,
		\notag \\
		\frac{\partial g^{\rm red}}{\partial \lambda_2}
		& =
		\frac{\partial g}{\partial \lambda_2} -
		\frac{\lambda_3}{\lambda_2}
		\frac{\partial g} {\partial \lambda_3}
		=
		(\tBiot^{\rm ext,0})_2-\frac{\lambda_3}{\lambda_2} (\tBiot^{\rm ext,0})_3\,.
		\label{eq:red-derivative-2}
	\end{align}
	Hence the reduced energy does not determine the three ambient principal Biot
	stresses $(\tBiot^{\rm ext,0})_1, (\tBiot^{\rm ext,0})_2$ and $(\tBiot^{\rm ext,0})_3$ separately. It determines only the two combinations
	appearing in \eqref{eq:red-derivative-2}. Equivalently,
	\begin{align}
		(\tBiot^{\rm ext,0})_1
		& =
		\frac{\partial g^{\rm red}}{\partial \lambda_1}
		+
		\frac{\lambda_3}{\lambda_1} (\tBiot^{\rm ext,0})_3\,,
		\qquad
		(\tBiot^{\rm ext,0})_2
		=
		\frac{\partial g^{\rm red}}{\partial \lambda_2}
		+
		\frac{\lambda_3}{\lambda_2} (\tBiot^{\rm ext,0})_3\,.
		\label{eq:T2-from-gred}
	\end{align}
	Thus the ambient tensor $\D_U\What(U)$ restricted to $\mathcal M$ still
	depends on the chosen extension of the incompressible energy away from
	$\mathcal M$, through the quantity
	\begin{align}
		(\tBiot^{\rm ext,0})_3=
		\frac{ \partial g}{\partial \lambda_3}
		\left(\lambda_1,\lambda_2,\frac{1}{\lambda_1\lambda_2}\right).
	\end{align}
	Let us now change the space on which we are working, i.e.\ let us consider the change of variables $x_i\colonequals \log\lambda_i, \ i=1,2,3.$ On $\mathcal M$ one has $x_1+x_2+x_3=0, \ x_3=-x_1-x_2.$ Define
	\begin{align}
		\gtilde_{\rm red}(x_1,x_2)
		\colonequals
		g^{\rm red}(e^{x_1},e^{x_2})
		=
		g(e^{x_1},e^{x_2},e^{-x_1-x_2})\,.
	\end{align}
	Then
	\begin{align}
		\frac{\partial }{\partial x_1} \gtilde_{\rm red}
		=
		\lambda_1\ \frac{\partial g^{\rm red}}{\partial \lambda_1}\,,
		\qquad
		\frac{\partial }{\partial x_2} \gtilde_{\rm red}=
		\lambda_2\partial_{\lambda_2}g^{\rm red}\,.
	\end{align}
	Using \eqref{eq:red-derivative-2}, we obtain
	\begin{align}
		\frac{\partial }{\partial x_1} \gtilde_{\rm red}
		& =
		\lambda_1(\tBiotTilde^{\rm ext,0})_1-\lambda_3(\tBiotTilde^{\rm ext,0})_3\,,\qquad \frac{\partial }{\partial x_2} \gtilde_{\rm red}
		=
		\lambda_2(\tBiotTilde^{\rm ext,0})_2-\lambda_3(\tBiotTilde^{\rm ext,0})_3\,.
		\label{eq:phi-x2-kirchhoff-diff}
	\end{align}
	If $\tau_i\colonequals \lambda_i\.t_i$ denote the principal Kirchhoff stresses, then
	\eqref{eq:phi-x2-kirchhoff-diff} can be written
	as
	\begin{align}
		\D\gtilde_{\rm red}(x_1,x_2)
		=
		\bigl(\tau_1-\tau_3\,,\tau_2-\tau_3\bigr).
	\end{align}
	Thus the gradient of the logarithmic reduced energy determines the differences
	of the principal Kirchhoff stresses, not the differences of the principal Biot
	stresses.
	Due to the incompressibility constraint, the monotonicity of the principal Biot stress vector
	with respect to $x=\log\lambda$ means
	\begin{align}
		\iprod{\tBiotTilde^{\rm ext,0}(x)-\tBiotTilde^{\rm ext,0}(y)\,,x-y} \geq 0
		\label{eq:Biot-log-monotonicity-3D}
	\end{align}
	for all $x=(x_1,x_2,-x_1-x_2)\,, y=(y_1,y_2,-y_1-y_2).$
	Then, we have
	\begin{align}
		x-y=
		\bigl(\Delta x_1,\Delta x_2,
		-\Delta x_1-\Delta x_2\bigr),
	\end{align}
	where $\Delta x_\alpha\colonequals x_\alpha-y_\alpha, \ \alpha=1,2.$ Therefore,
	\begin{align*}
		\iprod{\tBiotTilde^{\rm ext,0}(x)-\tBiotTilde^{\rm ext,0}(y)\,,x-y}
		& =
		\big[((\tBiotTilde^{\rm ext,0})_1-(\tBiotTilde^{\rm ext,0})_3)(x)-((\tBiotTilde^{\rm ext,0})_1-(\tBiotTilde^{\rm ext,0})_3)(y)\big]\Delta x_1
		\\
		& \quad+
		\big[((\tBiotTilde^{\rm ext,0})_2-(\tBiotTilde^{\rm ext,0})_3)(x)-((\tBiotTilde^{\rm ext,0})_2-(\tBiotTilde^{\rm ext,0})_3)(y)\big]\Delta x_2\,.
	\end{align*}
	Consequently, the monotonicity condition
	\eqref{eq:Biot-log-monotonicity-3D} is equivalent to the monotonicity of the
	two-dimensional vector field
	\begin{align}
		(x_1,x_2)\mapsto
		\mathcal B(x_1,x_2)
		\colonequals
		\bigl((\tBiotTilde^{\rm ext,0})_1-(\tBiotTilde^{\rm ext,0})_3\,,(\tBiotTilde^{\rm ext,0})_2-(\tBiotTilde^{\rm ext,0})_3\bigr)^T.
	\end{align}
	Explicitly,
	\begin{align}
		&
		\big[\mathcal B_1(x)-\mathcal B_1(y)\big](x_1-y_1)
		+
		\big[\mathcal B_2(x)-\mathcal B_2(y)\big](x_2-y_2)
		\geq 0
	\end{align}
	for all $x,y\in\R^2$, where
	\begin{align}
		\mathcal B_1=(\tBiotTilde^{\rm ext,0})_1-(\tBiotTilde^{\rm ext,0})_3\,,
		\qquad\qquad
		\mathcal B_2=(\tBiotTilde^{\rm ext,0})_2-(\tBiotTilde^{\rm ext,0})_3\,.
	\end{align}
	If the stress field is $C^1$, this is locally equivalent to
	\begin{align}
		\sym\D_{(x_1,x_2)} \mathcal B\in \Sym^{+}(2)\,.
	\end{align}
	We can express the components of $\mathcal B$ in terms of the reduced energy
	and the remaining ambient quantity $t_3$. From
	\eqref{eq:red-derivative-2} one obtains
	\begin{align}
		(\tBiotTilde^{\rm ext,0})_1-(\tBiotTilde^{\rm ext,0})_3
		& =
		\frac{\partial g^{\rm red}}{\partial \lambda_1}
		+
		\left(\frac{\lambda_3}{\lambda_1}-1\right)(\tBiotTilde^{\rm ext,0})_3\,,
		\label{eq:T1-minus-T3-gred}
		\\
		(\tBiotTilde^{\rm ext,0})_2-(\tBiotTilde^{\rm ext,0}) _3
		& =
		\frac{\partial g^{\rm red}}{\partial \lambda_2}
		+
		\left(\frac{\lambda_3}{\lambda_2}-1\right)(\tBiotTilde^{\rm ext,0})_3\,.
		\label{eq:T2-minus-T3-gred}
	\end{align}
	The preceding formulas show the essential difference between the ambient Biot
	stress and the stress response determined intrinsically on the incompressible
	manifold. The reduced energy $g^{\rm red}$ does not determine the vector
	\begin{align}
		( (\tBiotTilde^{\rm ext,0})_1- (\tBiotTilde^{\rm ext,0})_3\,, (\tBiotTilde^{\rm ext,0})_2- (\tBiotTilde^{\rm ext,0})_3)
	\end{align}
	of principal zero-pressure extrinsic Biot stress differences unless the ambient extension, or equivalently
	the value of $t_3$, is specified. Hence logarithmic monotonicity of
	$\TBiot^{\rm ext,0}=\D_U\What$ on $\mathcal M$ is not a condition
	expressible solely in terms of the reduced energy $g^{\rm red}$. Consequently, the monotonicity of the zero-pressure extrinsic Biot-stress tensor cannot in general be expressed solely in terms of the reduced energy $g^{\rm red}$. It also depends on the chosen ambient extension of the energy away from the incompressibility constraint. By contrast, the corresponding logarithmic monotonicity condition for the Kirchhoff stress is determined entirely by the reduced logarithmic energy:
	\begin{align}
		\D_{(x_1,x_2)} \gtilde_{\rm red} = \bigl(\tau_1-\tau_3\,,\tau_2-\tau_3\bigr),
	\end{align}
	and therefore local monotonicity of the Kirchhoff stress in logarithmic principal variables is simply
	\begin{align}
		\D^2_{(x_1,x_2)} \gtilde_{\rm red}(x_1,x_2)\in \Sym^{+}(2)\,.
	\end{align}
	However, for the principal zero-pressure extrinsic Biot stress,
	we still can write this condition in terms of the Hessian of the ambient energy. Let
	\begin{align}
		H\colonequals \D^2_\lambda g(\lambda_1,\lambda_2,\lambda_3) = \left(\partial^2_{\lambda_i\lambda_j}g\right)_{i,j=1}^3.
	\end{align}
	Furthermore, set
	\begin{align}
		A\colonequals  \matr{1 & 0 & -1 \\ 0&1&-1}
		\qquad\implies\qquad
		B = \matr{(\tBiotTilde^{\rm ext,0})_1-(\tBiotTilde^{\rm ext,0})_3 \\ (\tBiotTilde^{\rm ext,0})_2-(\tBiotTilde^{\rm ext,0})_3} = A\.\D_\lambda g\,.
	\end{align}
	Since $\lambda_1=e^{x_1}$, $\lambda_2=e^{x_2}$ and $\lambda_3=e^{-x_1-x_2}$,
	we have
	\begin{align}
		J_x\colonequals \D_x\lambda = \matr{\lambda_1 & 0 \\ 0&\lambda_2\\ -\lambda_3&-\lambda_3}.
	\end{align}
	Therefore, by the chain rule,
	\begin{align}
		\D_xB = A\.\D^2_\lambda g\.J_x.
	\end{align}
	Equivalently,
	\begin{align}
		\D_x \matr{(\tBiotTilde^{\rm ext,0})_1-(\tBiotTilde^{\rm ext,0})_3 \\ (\tBiotTilde^{\rm ext,0})_2-(\tBiotTilde^{\rm ext,0})_3} = A\.H\.J_x.
	\end{align}
	Hence the local logarithmic monotonicity of the ambient Biot stress is given by
	\begin{align}
		\sym\.\bigl(A\.\D^2_\lambda g\.J_x\bigr)\in {\Sym}^{+}(2)\,.
	\end{align}
%
%
	\subsection{Monotonicity of the extrinsic Biot stress tensor representative with vanishing third principal stress -- plane stress}\label{smB}
	Let us now consider the pressure-dependent Biot-type representative, i.e.\ the case of the full extrinsic Biot stress tensor given by
	\begin{align}
		\TBiot^{\rm ext}(U)
		=
		\D_U\What(U)-p\,U^{-1}.
	\end{align}
	In principal coordinates its values are
	\begin{align}
		(\tBiot^{\rm ext})_i
		=
		\frac{\partial g}{\partial \lambda_i}-\frac{p}{\lambda_i},
		\qquad i=1,2,3\,.
	\end{align}
	We impose the normalization (plane stress), i.e.\ $(\tBiot^{\rm ext})_3=0$.
	This fixes the pressure and thereby selects the plane-stress representative $\TBiot^{\rm ps}$.
	This gives $0= \frac{\partial g}{\partial \lambda_3}-\frac{p}{\lambda_3},$ and therefore $p=\lambda_3 \.\frac{\partial g}{\partial \lambda_3}$.
	Consequently,
	\begin{align}
		(\tBiot^{\rm ps})_1
		=
		\frac{\partial g}{\partial \lambda_1}-\frac{\lambda_3}{\lambda_1} \frac{\partial g}{\partial \lambda_3}\,,
		\qquad 
		(\tBiot^{\rm ps})_2
		=
		\frac{\partial g}{\partial \lambda_2}-\frac{\lambda_3}{\lambda_2} \frac{\partial g}{\partial \lambda_3}\,,
		\qquad 
		(\tBiot^{\rm ps})_3=0\,.
	\end{align}
	On the incompressibility constraint $\lambda_1\lambda_2\lambda_3=1,$ we define the reduced energy
	\begin{align}
		g^{\rm red}(\lambda_1,\lambda_2)
		\colonequals
		g\left(\lambda_1,\lambda_2,
		\frac{1}{\lambda_1\lambda_2}\right),
	\end{align}
	and by the chain rule one obtains
	\begin{align}
		(\tBiot^{\rm ps})_1
		=
		\frac{\partial g^{\rm red}}{\partial \lambda_1}\,,
		\qquad \qquad
		(\tBiot^{\rm ps})_2
		=
		\frac{\partial g^{\rm red}}{\partial \lambda_2}\,,
		\qquad \qquad
		(\tBiot^{\rm ps})_3=0\,.
	\end{align}
	Thus, once the distinguished third direction is fixed, this representative is determined by the reduced energy $g^{\rm red}$.
	The logarithmic monotonicity of the principal Biot stress vector is equivalent to the logarithmic monotonicity of the two-dimensional vector field
	\begin{align}
		\matr{(\tBiot^{\rm ps})_1 \\ (\tBiot^{\rm ps})_2},
	\end{align}
	i.e.
	the monotonicity condition is locally equivalent to
	\begin{align}
		\sym\D_{(\log\lambda_1,\log\lambda_2)}
		\bigl(\D_{(\lambda_1,\lambda_2)} g^{\rm red}\bigr)\in \Sym^{+}(2)\,.
	\end{align}
	We now write the corresponding tensor in invariant form. Let
	\begin{align}
		U=\sum_{i=1}^3\lambda_i N_i\,,
		\qquad
		N_i=q_i\otimes q_i\,,
	\end{align}
	be the spectral decomposition of $U$, with $N_3$ denoting the projector
	associated with the distinguished third principal direction. Since
	\begin{align}
		\D_U\What(U)U
		=
		\sum_{i=1}^3 \lambda_i \frac{\partial g}{\partial \lambda_i}N_i\,,
	\end{align}
	we have $\lambda_3\.\frac{\partial g}{\partial \lambda_3}=\iprod{\D_U\What(U)U,N_3}\,.$
	Thus the representative selected by $(\tBiot^{\rm ps})_3=0$ is
	\begin{align}
		\TBiot^{\rm ps}(U)
		=
		\D_U\What(U)
		-
		\iprod{\D_U\What(U)U,N_3} U^{-1}.
	\end{align}
	Therefore, for the reduced two-dimensional principal response selected by the plane-stress normalization, logarithmic vector monotonicity is locally equivalent to
	\begin{align}
		\sym\D_{(\log\lambda_1,\log\lambda_2)}
		\bigl(\D_{(\lambda_1,\lambda_2)} g^{\rm red}\bigr)\in \Sym^{+}(2)\,.
	\end{align}
	The normalization $(\tBiot^{\rm ps})_3=0$ singles out the third principal direction. Consequently, without additional structure fixing this direction, $\TBiot^{\rm ps}$ is not a permutation-equivariant isotropic tensor field on the full manifold $\mathcal M$, and Proposition \ref{prop:spectral-equivalence} does not apply directly to this representative.
	%
	%
	%
	%
	\section{Conclusion}
	We have discussed tensor and vector monotonicity concepts, with a focus on their consequences for compressible and incompressible isotropic constitutive laws. The focus on properties of the Biot-stress tensor is motivated solely by its current use in formulating constitutive restrictions in nonlinear elasticity. 
	The logarithmic monotonicity notion on the incompressible stretch manifold
	$\mathcal M=\Symp(3)\cap\SL(3)$ is naturally related to the logarithmic strain.
	Logarithmic strain allows for a transfer of the monotonicity framework from the compressible to the incompressible case.
	Since $\TBiot(U)=\D_U \What(U)$, {even in the compressible case}, the monotonicity of $\TBiot$ in terms of $\log U$ is not a convexity requirement on the energy expressed as a function of $\log U$.
	Rather, {in the compressible case}, one has to check
	\begin{align}
		\label{con}
		\sym \D_{\log U} & (\D_U\What(U))\in \Sym_4^{++}(6) \notag                                                                                                                                       \\
		& \iff \quad \sym \D_{(\log\lambda_1,\log\lambda_2,\log\lambda_3) }(\D_{(\lambda_1,\lambda_2,\lambda_3) }g((\lambda_1,\lambda_2,\lambda_3) ))\in \Sym^{++}(3)\,.
	\end{align}
	Strict logarithmic monotonicity implies injectivity of the corresponding load--stretch relation. In the incompressible biaxial setting, the analogous reduced condition therefore excludes multiple stretch states corresponding to the same reduced dead load, as occurring in a Treloar-Kearsley type instability.
	All quantities, concepts, constitutive requirements and results have to be formulated directly within the incompressible setting.
	For incompressible materials we consider the same question for the extension-dependent zero-pressure extrinsic Biot-stress tensor $\TBiot^{\rm ext,0}$.
	After logarithmic reparametrization, the matrix--vector equivalence follows from Theorem \ref{theorem:mainResult}.
	In contrast, the intrinsic representative $\TBiot^{\rm inc}$ is independent of the chosen ambient extension, while $\TBiot^{\rm ps}$ corresponds to the separate plane-stress normalization obtained after fixing a distinguished third direction.
	Moreover, we prove that for the plane-stress normalization, characterized by a vanishing third principal Biot stress, the monotonicity is ensured by
	\begin{align}
		\label{110c}
		\sym \D_{(\log \lambda_1,\log \lambda_2)} & [\D_{(\lambda_1,\lambda_2)} g^{\rm red}( \lambda_1, \lambda_2)] {=\sym \. [(\D^2_{(\lambda_1,\lambda_2)} g^{\rm red}( \lambda_1, \lambda_2))\diag (\lambda_1,\lambda_2)]}\in \Symp(2)\,,
	\end{align}
	i.e.\ an analogous form of the condition considered in the compressible case. Note that \eqref{110c} is neither the convexity of the map $(\lambda_1,\lambda_2)\mapsto g^{\rm red}(\lambda_1,\lambda_2)$ nor geodesic convexity, see Appendix \ref{geoc}.
	It has become apparent that monotonicity conditions on the Biot stress $\TBiot$ in incompressible nonlinear elasticity are quite awkward since the non-constitutive indeterminate pressure term $p\,U^{-1}$ always intervenes.
	In stark contrast, the logarithmic monotonicity condition for the Cauchy-stress tensors $\sigma$, i.e.
	\begin{align}
		\label{sm}
		\iprod{\sigmahat(\log V_1)- \sigmahat(\log V_2)\,, \log V_1-\log V_2} >0 \qquad \forall\; V_1,V_2\in \Symp(n)\cap \SL(n)\,,
	\end{align}
	immediately cancels the indeterminate pressure, since $p\.\id$ is orthogonal to $\tr(\log V)=0$, and reduces to Hill's inequality for the Kirchhoff stress $\tau$ \cite{agn_ghiba2026polyconvexity,agn_wollner2026concurrent}.
	Therefore, condition \eqref{sm}, which has recently been re-proposed \cite{agn_neff2026rate,neff2024corotational,neff2024natural,d2024constitutive}, appears as a natural constitutive requirement fully consistent with the incompressible limit.
	%
	%
	%
	%
	\section*{Acknowledgement}
		The work of I.D. Ghiba  was supported by a grant of the Ministry of Education and Research, CNCS--UEFISCDI,
		project number PN-IV-P1-PCE2026-0163, within PNCDI IV.
	
	\footnotesize
	\section{References}
	\printbibliography[heading=none]

@article{anssari2023large,
  title={Large isotropic elastic deformations: on a comprehensive model to correlate the theory and experiments for incompressible rubber-like materials},
  author={Anssari-Benam, Amin},
  journal={Journal of Elasticity},
  volume={153},
  number={2},
  pages={219--244},
  year={2023}
}

@article{anssari2022modelling,
  title={Modelling brain tissue elasticity with the {Ogden} model and an alternative family of constitutive models},
  author={Anssari-Benam, Amin and Destrade, Michel and Saccomandi, Giuseppe},
  journal={Philosophical Transactions of the Royal Society A: Mathematical, Physical and Engineering Sciences},
  volume={380},
  number={2234},
  year={2022}
}

@article{baaser2026hyperelastic,
  title={Hyperelastic stability landscape: A check for {Hill} stability of isotropic, incompressible hyperelasticity depending on material parameters},
  author={Baaser, Herbert},
  journal={Journal of Elasticity},
  volume={158},
  number={1},
  pages={8},
  year={2026}
}

@article{batra2005treloar,
  title={Treloar's biaxial tests and {Kearsley's} bifurcation in rubber sheets},
  author={Batra, Romesh C. and Mueller, Ingo and Strehlow, Peter},
  journal={Mathematics and Mechanics of Solids},
  volume={10},
  number={6},
  pages={705--713},
  year={2005}
}

@article{birkhoff1946three,
  title={Three observations on linear algebra},
  author={Birkhoff, Garrett},
  journal={Revista de la Universidad Nacional de Tucum{\'a}n, Serie A},
  volume={5},
  pages={147--151},
  year={1946}
}

@book{DoCarmo,
  title={Differential Geometry of Curves and Surfaces},
  author={do Carmo, Manfredo P.},
  publisher={Prentice Hall},
  address={Englewood Cliffs, New Jersey},
  year={1976}
}

@book{hardy1952inequalities,
  title={Inequalities},
  author={Hardy, Godfrey H. and Littlewood, John E. and P{\'o}lya, George},
  publisher={Cambridge University Press},
  year={1952}
}

@article{hill1968constitutiveb,
  title={On constitutive inequalities for simple materials---{II}},
  author={Hill, Rodney},
  journal={Journal of the Mechanics and Physics of Solids},
  volume={16},
  number={5},
  pages={315--322},
  year={1968}
}

@article{horn1954doubly,
  title={Doubly stochastic matrices and the diagonal of a rotation matrix},
  author={Horn, Alfred},
  journal={American Journal of Mathematics},
  volume={76},
  number={3},
  pages={620--630},
  year={1954}
}

@book{MarshallOlkinArnold,
  title={Inequalities: Theory of Majorization and Its Applications},
  author={Marshall, Albert W. and Olkin, Ingram and Arnold, Barry C.},
  series={Springer Series in Statistics},
  publisher={Springer},
  year={2011}
}

@incollection{ogden2001elements,
  title={Elements of the theory of finite elasticity},
  author={Ogden, Raymond W.},
  booktitle={Nonlinear Elasticity: Theory and Applications},
  year={2001}
}

@article{carlson1986derivative,
	author  = {Carlson, Donald E. and Hoger, Anne},
	title   = {The derivative of a tensor-valued function of a tensor},
	journal = {Quarterly of Applied Mathematics},
	volume  = {44},
	number  = {3},
	pages   = {409--423},
	year    = {1986}
}

@article{chadwick1971theorem,
	author  = {Chadwick, P. and Ogden, R. W.},
	title   = {A theorem of tensor calculus and its application to isotropic elasticity},
	journal = {Archive for Rational Mechanics and Analysis},
	volume  = {44},
	pages   = {54--68},
	year    = {1971}
}

@article{agn_neff2014rediscovering,
	title={Rediscovering {G}.\,{F}.\ {B}ecker's early axiomatic deduction of a multiaxial nonlinear stress--strain relation based on logarithmic strain},
	author={Neff, Patrizio and M{\"u}nch, Ingo and Martin, Robert J.},
	journal={Mathematics and Mechanics of Solids},
	year={2016},
	publisher={SAGE Publications},
	doi={10.1177/1081286514542296},
	volume={21},
	number={7},
	pages={856--911},
	class =	{mathematics}
}

@article{agn_neff2014grioli,
	title={On Grioli's minimum property and its relation to {C}auchy's polar decomposition},
	author={Neff, Patrizio and Lankeit, Johannes and Madeo, Angela},
	journal={International Journal of Engineering Science},
	volume={80},
	pages={209--217},
	year={2014},
	publisher={Elsevier},
	doi={10.1016/j.ijengsci.2014.02.026},
	class =	{matrix}
}

@article{agn_lehmich2013convexity,
	title={On the convexity of the function {$C\mapsto f (\det C)$} on positive-definite matrices},
	author={Lehmich, Stephan and Neff, Patrizio and Lankeit, Johannes},
	journal={Mathematics and Mechanics of Solids},
	pages={1081286512466099},
	year={2013},
	publisher={SAGE Publications},
	doi={10.1177/1081286512466099},
	class =	{matrix}
}

@article{agn_martin2015some,
	title={Some remarks on the monotonicity of primary matrix functions on the set of symmetric matrices},
	author={Martin, Robert J. and Neff, Patrizio},
	journal={Archive of Applied Mechanics},
	volume={85},
	number={12},
	pages={1761--1778},
	year={2015},
	publisher={Springer},
	doi={10.1007/s00419-015-1017-4},
	class =	{matrix}
}

@article{agn_neff2015exponentiatedI,
	title={The exponentiated {H}encky-logarithmic strain energy. {P}art {I}: {C}onstitutive issues and rank-one convexity},
	author={Neff, Patrizio and Ghiba, Ionel-Dumitrel and Lankeit, Johannes},
	journal={Journal of Elasticity},
	volume={121},
	number={2},
	pages={143--234},
	year={2015},
	publisher={Springer},
	doi={10.1007/s10659-015-9524-7},
	class =	{mathematics}
}

@article{agn_neff2015geometry,
	title =		{Geometry of logarithmic strain measures in solid mechanics},
	author =	{Neff, Patrizio and Eidel, Bernhard and Martin, Robert J.},
	journal =	{Archive for Rational Mechanics and Analysis},
	volume =	{222},
	number =	{2},
	pages =		{507--572},
	note =		{\availableatarxiv{1505.02203}},
	year =		{2016},
	doi =		{10.1007/s00205-016-1007-x},
	class =	{mathematics}
}

@article{agn_thiel2019empirical,
  title =	{Do we need {T}ruesdell's empirical inequalities? {O}n the coaxiality of stress and stretch},
  author =	{Thiel, Christian and Voss, Jendrik and Martin, Robert J. and Neff, Patrizio},
  journal =	{International Journal of Non-Linear Mechanics},
  year =	{2019},
  publisher={Elsevier},
  note =	{\availableatarxiv{1812.03053}},
  doi=		{10.1016/j.ijnonlinmec.2019.02.004},
  class =	{mathematics}
}

@article{agn_neff2020axiomatic,
  title={The axiomatic introduction of arbitrary strain tensors by Hans Richter -- a commented translation of \enquote{Strain tensor, strain deviator and stress tensor for finite deformations}},
  author={Neff, Patrizio and Graban, Kai and Schweickert, Eva and Martin, Robert J},
  journal={Mathematics and Mechanics of Solids},
  volume={25},
  number={5},
  pages={1060--1080},
  year={2020},
  publisher={Sage Publications},
  note = {\availableatarxiv{1901.00058}}
}

@article{agn_voss2019volIsLog,
  title =	{Sharp rank-one convexity conditions in planar isotropic elasticity for the additive volumetric-isochoric split},
  author =	{Voss, Jendrik and Ghiba, Ionel-Dumitrel and Martin, Robert J. and Neff, Patrizio},
  journal={Journal of Elasticity},
  volume={143},
  number={2},
  pages={301--335},
  year={2021},
  publisher={Springer},
  note={\availableatarxiv{2008.04188}},
}

@article{neff2024corotational,
  title={The corotational stability postulate: positive incremental Cauchy stress moduli for diagonal, homogeneous deformations in isotropic nonlinear elasticity},
  author={Neff, Patrizio and Husemann, Nina J and Tchakoutio, Aur{\'e}lien S Nguetcho and Korobeynikov, Sergey N and Martin, Robert J},
  journal={International Journal of Non-Linear Mechanics},
  volume={174},
  pages={105033},
  year={2025},
  publisher={Elsevier},
  note={\availableatarxiv{2411.12552}}
}

@article{neff2024hypo,
  title={Hypo-elasticity, Cauchy-elasticity, corotational stability and monotonicity in the logarithmic strain},
  author={Neff, Patrizio and Holthausen, Sebastian and d'Agostino, Marco Valerio and Bernardini, Davide and Sky, Adam and Ghiba, Ionel-Dumitrel and Martin, Robert J},
  journal={Journal of the Mechanics and Physics of Solids},
  volume={202},
  pages={106074},
  year={2025},
  publisher={Elsevier},
  note={\availableatarxiv{2409.20051}}
}

@article{neff2024natural,
  title={A natural requirement for objective corotational rates--on structure preserving corotational rates},
  author={Neff, Patrizio and Holthausen, Sebastian and Korobeynikov, Sergey N and Ghiba, Ionel-Dumitrel and Martin, Robert J},
  journal={submitted},
  note={\availableatarxiv{2409.19707}},
  year={2024}
}

@article{d2024constitutive,
  title={A constitutive condition for idealized isotropic Cauchy elasticity involving the logarithmic strain},
  author={d'Agostino, Marco Valerio and Holthausen, Sebastian and Bernardini, Davide and Sky, Adam and Ghiba, Ionel-Dumitrel and Martin, Robert J and Neff, Patrizio},
  journal={Journal of Elasticity},
  volume={157},
  number={1},
  pages={23},
  year={2025},
  publisher={Springer},
  note={\availableatarxiv{2409.01811}}
}

@article{agn_ghiba2025biot,
	title={The Biot stress-right stretch relation for the compressible Neo-Hooke-Ciarlet-Geymonat model and Rivlin's cube problem},
	author={Ghiba, Ionel-Dumitrel and Gmeineder, Franz and Holthausen, Sebastian and Martin, Robert J. and Neff, Patrizio},
	journal={Communications in Mathematical Analysis and Applications},
	volume={4},
	number={1},
	pages={112--150},
	year={2025}
}

@article{agn_ghiba2026polyconvexity,
	title={Polyconvexity implies Hill's inequality in $\mathrm{SL}(2)$},
	author={Ghiba, Ionel-Dumitrel and Wollner, Maximilian P. and Neff, Patrizio},
	journal={European Journal of Mechanics - A/Solids},
	volume={121},
	pages={106296},
	year={2026},
	publisher={Elsevier}
}

@article{agn_voss2026monotonicity,
	title={Monotonicity of isotropic tensor functions on the set of symmetric matrices: completing Rodney Hill's generalization of the Chandler Davis convexity theorem},
	author={Voss, Jendrik and Martin, Robert J. and Ghiba, Ionel-Dumitrel and d'Agostino, Marco Valerio and Neff, Patrizio},
	journal = {to appear in European Journal of Mechanics - A/Solids},
	note =	{\availableatarxiv{2608.07087}},
	year={2026}
}

@article{agn_neff2026rate,
	title={Rate-form equilibrium for an isotropic Cauchy-elastic formulation: Part I: modeling},
	author={Neff, Patrizio and Husemann, Nina J. and Holthausen, Sebastian and Gmeineder, Franz and Blesgen, Thomas},
	journal={Journal of Nonlinear Science},
	volume={36},
	number={1},
	pages={8},
	year={2026},
	publisher={Springer}
}

@article{agn_wollner2026concurrent,
	title={Concurrent enforcement of polyconvexity and true-stress-true-strain monotonicity in incompressible isotropic hyperelasticity: application to neural network constitutive models},
	author={Wollner, Maximilian P. and Klein, Dominik K. and Baaser, Herbert and Holzapfel, Gerhard A. and Neff, Patrizio},
	journal={Journal of the Mechanics and Physics of Solids},
	year={2026},
	note={in press}
}

@article{hill1970,
  title =	{Constitutive inequalities for isotropic elastic solids under finite strain},
  author =	{R. Hill},
  year =		{1970},
  journal =	{Proceedings of the Royal Society of London. A. Mathematical and Physical Sciences},
  volume =	{314},
  pages =	{457--472}
}

@article{richter1948,
  year =			{1948},
  journal =		{Zeitschrift f{\"u}r Angewandte Mathematik und Mechanik},
  volume =		{28},
  number =		{7/8},
  title =		{Das isotrope {E}lastizit{\"a}tsgesetz},
  author =		{Richter, H.},
  pages =		{205--209},
  note =		{\availableaturl[uni-due.de]{https://www.uni-due.de/imperia/md/content/mathematik/ag_neff/richter_isotrop_log.pdf}}
}

@book{Ogden83,
  author = 		{R. W. Ogden},
  title =		{Non-{L}inear {E}lastic {D}eformations.},
  edition =		{1},
  publisher =	{Ellis Horwood},
  series =		{Mathematics and its Applications},
  year =		{1983},
  address =		{Chichester}
}

@article{becker1893,
  author =	{Becker, G. F.},
  title =	{The finite elastic stress-strain function},
  journal =	{American Journal of Science},
  volume =	{46},
  pages =	{337--356},
  year =	{1893},
  note =	{newly typeset version \availableaturl[uni-due.de]{https://www.uni-due.de/imperia/md/content/mathematik/ag_neff/becker_latex_new1893.pdf}}
}

@incollection{truesdell65,
  author =		{C. Truesdell and W. Noll},
  title =		{The Non-Linear Field Theories of Mechanics},
  booktitle =	{Handbuch der {P}hysik},
  page =		{1--591},
  publisher=	{Springer},
  volume =		{III/3},
  editor =		{S. Fl\"ugge},
  year =		{1965},
  address=		{Heidelberg}
}

@article {richter1949verzerrung,
  author =		{Richter, H.},
  title =		{Verzerrungstensor, {V}erzerrungsdeviator und {S}pannungstensor bei endlichen {F}orm\-{\"a}nderungen},
  journal =		{Zeitschrift f{\"u}r Angewandte Mathematik und Mechanik},
  volume =		{29},
  number =		{3},
  issn =		{1521-4001},
  pages =		{65--75},
  year =		{1949},
  note =		{\availableaturl[uni-due.de]{https://www.uni-due.de/imperia/md/content/mathematik/ag_neff/richter_deviator_log.pdf}}
}

@article{ball1976convexity,
  title =		{Convexity conditions and existence theorems in nonlinear elasticity},
  author =		{Ball, J. M.},
  journal =		{Archive for Rational Mechanics and Analysis},
  volume =		{63},
  number =		{4},
  pages =		{337--403},
  year =		{1976},
  publisher =	{Springer}
}

@article{Hill68,
  author =	{R. Hill},
  title =	{On constitutive inequalities for simple materials - {I}},
  journal =	{Journal of the Mechanics and Physics of Solids},
  volume=	{11},
  pages=	{229-242},
  year =	{1968}
}

@article{coleman1959thermostatics,
  title={On the thermostatics of continuous media},
  author={Coleman, Bernard D and Noll, Walter},
  journal={Archive for Rational Mechanics and Analysis},
  volume={4},
  number={1},
  pages={97--128},
  year={1959},
  publisher={Springer}
}
	\small
	\appendix
	%
	%
	%
	\section{Extrinsic versus\ intrinsic Biot-stress tensor for incompressible materials}\label{eiA}
	Let us return to the full form of the Biot-stress tensor for an incompressible material. 
	We consider an isotropic energy $\What\col \Symp(n)\to\R$ in terms of the right stretch tensor $U$ and its restriction
	\begin{align}
		\What^{\rm inc}\col \mathcal M\to\R\,,
		\qquad	\mathcal M=\{U\in\Symp(n)\mid \det U=1\}\,.
	\end{align}
	The tangent space of $\mathcal M$ at $U\in\mathcal M$ is given by
	\begin{align}
		T_U\mathcal M=\left\{H\in\Sym(n)\,\middle|\,\tr(U^{-1}H)=0\right\}.\,
	\end{align}
	Thus, $U^{-1}$ is normal to the tangent space in the sense that
	\begin{align}
		\iprod{U^{-1},H}=\tr(U^{-1}H^T)=\tr(U^{-1}H)=0
		\qquad	\forall\; H\in T_U\mathcal M\,.
	\end{align}
	Consequently, if $T\in\Sym(n)$ represents the differential of the constrained energy at $U$, then so does $T-p\,U^{-1}$ for any $p\in\R$, since
	\begin{align}
		\iprod{T-p\,U^{-1},H}
		=\iprod{T,H}
		\qquad\forall H\in T_U\mathcal M\,.
	\end{align}
	\textbf{This is the origin of the arbitrary pressure $p$ in incompressible elasticity.}
	In particular, the \emph{extrinsic Biot-stress representatives} are of the form
	\begin{align}
		\TBiot^{\rm ext}(U)\colonequals\D_U\big[\What(U)-p\,(\det U-1)\big]
		=\D_U\What(U)-p\,U^{-1},
		\label{eq:extrinsic-Biot-general}
	\end{align}
	where $p$ is an arbitrary scalar Lagrange multiplier, cf.\ \cite[Eq.~(2.81)]{ogden2001elements}.
	Hence, the incompressible energy does not determine one unique ambient Biot-stress tensor, but rather the family
	\begin{align}
		\left\{\D_U\What(U)-p\,U^{-1}\,\middle|\,p\in\R\right\}.
		\label{eq:Biot-equivalence-class}
	\end{align}
	In terms of the logarithmic strain we define $\Wtilde(X)\colonequals \What(\exp X)$ for $X\in\Sym(n)$ and the corresponding incompressibility constraint $\det U=1$ is equivalent to $\tr(\log U)=0$, so that
	\begin{align}
		\log U\in\dev\Sym(n)
		\qquad\forall\; U\in\mathcal M\,.
	\end{align}
	Accordingly, we define the restricted energy
	\begin{align}
		\Wtilde^{\rm inc}\col\dev\Sym(n)\to\R,
		\qquad	\Wtilde^{\rm inc}(X)=\Wtilde(X)
		\quad\forall\; X\in\dev\Sym(n)\,.
	\end{align}
	If $\Wtilde^{\rm inc}$ is differentiable, then its derivative at $X\in\dev\Sym(n)$ is a linear functional on $\dev\Sym(n)$.
	Since $\dev\Sym(n)$ is a finite-dimensional inner-product space, every linear functional on $\dev\Sym(n)$ has a unique representation with respect to the Frobenius inner product. Hence, there exists a unique tensor $\widetilde\tau^{\rm inc}(X)\in\dev\Sym(n)$ such that
	\begin{align}
		\iprod{\widetilde\tau^{\rm inc}(X)\,,H}
		=\D_X\Wtilde^{\rm inc}(X).H
		\qquad	\forall\; H\in\dev\Sym(n)\,.
		\label{eq:intrinsic-Kirchhoff-Riesz}
	\end{align}
	We call $\widetilde\tau^{\rm inc}$ the \emph{intrinsic Kirchhoff stress tensor}.
	Since $\Wtilde^{\rm inc}$ is the restriction of $\Wtilde$ to $\dev\Sym(n)$, we have
	\begin{align}
		\D_X\Wtilde^{\rm inc}(X).H
		=\iprod{\D_X\Wtilde(X)\,,H}
		=\iprod{\dev\bigl(\D_X\Wtilde(X)\bigr)\,,H}
	\end{align}
	for every $H\in\dev\Sym(n)$. Therefore,
	\begin{align}
		\widetilde\tau^{\rm inc}(X) 
		= \dev\bigl(\D_X\Wtilde(X)\bigr)
		=\D_X\Wtilde(X)-\frac{1}{n}\tr\bigl(\D_X\Wtilde(X)\bigr)\id\,.
		\label{eq:intrinsic-Kirchhoff-dev}
	\end{align}
	On the other hand, every tensor of the form
	\begin{align}
		\widetilde\tau^{\rm ext}(X) = \D_X\Wtilde(X)-p\.\id
		\label{eq:extrinsic-Kirchhoff-family}
	\end{align}
	with arbitrary $p\in\R$ satisfies
	\begin{align}
		\iprod{\widetilde\tau^{\rm ext}(X)\,,H}
		=\D_X\Wtilde^{\rm inc}(X).H
		\qquad\forall\; H\in\dev\Sym(n)\,,
	\end{align}
	since $\iprod{\id,H}=\tr H=0$.
	Thus, all tensors in \eqref{eq:extrinsic-Kirchhoff-family} represent the same constrained first variation.
	The \emph{intrinsic Kirchhoff stress tensor} is obtained by choosing the unique representative which is itself deviatoric.
	Indeed, the condition $\tr\bigl(\widetilde\tau^{\rm ext}(X)\bigr)=0$ yields the particular pressure
	\begin{align}
		p_{\rm inc}(X)
		=\frac{1}{n}\tr\bigl(\D_X\Wtilde(X)\bigr)
		\qquad\implies\qquad
		\widetilde\tau^{\rm inc}(X)	= \D_X\Wtilde(X)-p_{\rm inc}(X)\id.
		\label{eq:p-intrinsic}
	\end{align}
	Hence, the intrinsic stress should be understood as a particular, purely constitutive representative of the equivalence class \eqref{eq:extrinsic-Kirchhoff-family}: it is obtained by removing the spherical part of the Kirchhoff stress.
	For $X=\log U$, the corresponding \emph{intrinsic Biot-stress tensor} is defined by
	\begin{align}
		\TBiot^{\rm inc}(U)
		=	\widetilde\tau^{\rm inc}(\log U)\,U^{-1}.
		\label{eq:intrinsic-Biot-from-Kirchhoff}
	\end{align}
	Isotropy implies that $\D_U\What(U)$ is coaxial \cite{agn_thiel2019empirical} with $U$ and therefore also with $\log U$. Hence, for $X=\log U$, the chain rule for $\Wtilde(X)=\What(\exp X)$ reduces to $\D_X\Wtilde(X)=\D_U\What(U)\.U$.
	Thus in the isotropic case, $\D_{\log U}\Wtilde(\log U)\,U^{-1}=\D_U\What(U)$, and we obtain
	\begin{align}
		\TBiot^{\rm inc}(U)
		=\left[\D_{\log U}\Wtilde(\log U) - \frac{1}{n}	\tr\bigl(\D_{\log U}\Wtilde(\log U)\bigr)\id\right]U^{-1}
		=\D_U\What(U)-\frac{1}{n}\tr\bigl(\D_U\What(U)\,U\bigr)U^{-1}.
		\label{eq:intrinsic-Biot-final}
	\end{align}
	This representative is independent of the chosen ambient extension. Indeed, if two extensions $\What_1$ and $\What_2$ coincide on $\mathcal M$, then
	\begin{equation}
		\D_U\What_2(U)-\D_U\What_1(U)=c(U)\,U^{-1}
	\end{equation}
	for some scalar $c(U)$. The additional normal contribution cancels in \eqref{eq:intrinsic-Biot-final}.
	Thus, the intrinsic Biot-stress tensor corresponds to the pressure choice
	\begin{align}
		p_{\rm inc}(U)
		=\frac{1}{n}\tr\bigl(\D_U\What(U)\,U\bigr).
		\label{eq:p-intrinsic-Biot}
	\end{align}
	This choice should be distinguished from other possible representatives of the same constrained stress response.
	In particular, for $n=3$ we may choose a principal basis and write
	\begin{align}
		U=\diag(\lambda_1,\lambda_2,\lambda_3)\,,
		\qquad\D_{\log U}\Wtilde(\log U)
		=\diag(\tau_1,\tau_2,\tau_3)\,,
	\end{align}
	where $\tau_i$ denote the principal Kirchhoff stresses associated with an arbitrary ambient extension of the incompressible energy.
	Then the principal values of the extrinsic Biot-stress tensor are
	\begin{align}
		t_i^{\rm ext}
		=\frac{\tau_i-p}{\lambda_i}\,,
		\qquad i=1,2,3\,.
		\label{eq:extrinsic-Biot-principal}
	\end{align}
	For the intrinsic representative, \eqref{eq:p-intrinsic} gives
	\begin{align}
		p_{\rm inc}
		=\frac{\tau_1+\tau_2+\tau_3}{3}
		\qquad\implies\qquad
		\tau_i^{\rm inc}
		=\tau_i-\frac{\tau_1+\tau_2+\tau_3}{3}\,,
		\qquad i=1,2,3\,.
		\label{eq:intrinsic-principal-Kirchhoff}
	\end{align}
	In particular, $\tau_1^{\rm inc}+\tau_2^{\rm inc}+\tau_3^{\rm inc}=0$.
	The corresponding principal intrinsic Biot stresses are $t_i^{\rm inc}=\frac{\tau_i^{\rm inc}}{\lambda_i}$.
	
	A different representative is obtained by prescribing vanishing Biot stress in the third principal direction.
	If this direction is distinguished, e.g.\ as the thickness direction in a plane-stress setting, we impose $t_3=0$.
	According to \eqref{eq:extrinsic-Biot-principal}, this is equivalent to the pressure choice $p_{\rm ps}=\tau_3$.
	Hence,
	\begin{align}
		\matr{t_1^{\rm ps}\\t_2^{\rm ps}\\t_3^{\rm ps}}
		=\matr{\dfrac{\tau_1-\tau_3}{\lambda_1}\\\dfrac{\tau_2-\tau_3}{\lambda_2}\\0}.
		\label{eq:plane-stress-Biot-principal}
	\end{align}
	Equivalently, if $N_3$ denotes the projector onto the distinguished third principal direction, then
	\begin{align}
		\TBiot^{\rm ps}(U)
		=\D_U\What(U) - \iprod{\D_U\What(U)\,U,N_3}U^{-1}.
		\label{eq:plane-stress-Biot-tensor}
	\end{align}
	Note that this choice singles out one particular principal direction and should therefore not be confused with the intrinsic representative, which is defined without such a distinguished direction.
	
	We finally relate both representatives directly to the reduced incompressible energy.
	With $x_i=\log\lambda_i$ the incompressibility condition $\lambda_1\lambda_2\lambda_3=1$ becomes
	\begin{align}
		x_1+x_2+x_3=0
		\qquad\iff\qquad
		x_3=-x_1-x_2\,.
	\end{align}
	Let $\gtilde$ denote the representation of $\Wtilde$ in terms of the logarithmic principal stretches, i.e.	\begin{align}
		\gtilde(x_1,x_2,x_3)
		=\Wtilde\bigl(\diag(x_1,x_2,x_3)\bigr)\,,
	\end{align}
	and define the reduced energy by
	\begin{align}
		\gtilde_{\rm red}(x_1,x_2)
		=\gtilde(x_1,x_2,-x_1-x_2)\,.
		\label{eq:log-reduced-energy}
	\end{align}
	Since $\tau_i=\partial_{x_i}\gtilde(x_1,x_2,x_3)$, the chain rule yields
	\begin{align}
		\D\gtilde_{\rm red}(x_1,x_2)
		=\D_{(x_1,x_2)}\gtilde(x_1,x_2,-x_1-x_2)
		=\matr{\tau_1-\tau_3\\\tau_2-\tau_3}\equalscolon q.
		\label{eq:reduced-gradient-stress-differences}
	\end{align}
	Thus, the reduced energy does not determine the absolute values of all three ambient principal Kirchhoff stresses.
	Instead, it determines the two pressure-independent stress differences $q_1=\tau_1-\tau_3$ and $q_2=\tau_2-\tau_3$.
	Indeed, replacing all principal stresses by
	\begin{align}
		(\tau_1,\tau_2,\tau_3)\longmapsto(\tau_1-p,\tau_2-p,\tau_3-p)
	\end{align}
	leaves both differences $q_1$ and $q_2$ unchanged. However, once a particular representative is prescribed, the reduced energy determines that representative uniquely.
	For the plane-stress choice $p_{\rm ps}=\tau_3$, equation \eqref{eq:reduced-gradient-stress-differences} immediately yields
	\begin{align}
		\matr{\tau_1^{\rm ps}\\	\tau_2^{\rm ps}\\ \tau_3^{\rm ps}}
		=	\matr{q_1\\	q_2\\0}
		\qquad\implies\qquad
		\matr{t_1^{\rm ps}\\ t_2^{\rm ps}\\ t_3^{\rm ps}}
		=\matr{\tau_1^{\rm ps}/\lambda_1\\ \tau_2^{\rm ps}/\lambda_2\\	\tau_3^{\rm ps}/\lambda_3}
		=\matr{q_1/\lambda_1\\ q_2/\lambda_2\\	0}.
		\label{eq:plane-stress-from-reduced-energy}
	\end{align}
	On the other hand, for the intrinsic representative, we require
	\begin{align}
		\tau_1^{\rm inc}+\tau_2^{\rm inc}+\tau_3^{\rm inc}=0\,,
		\qquad		\tau_1^{\rm inc}-\tau_3^{\rm inc}=q_1\,,
		\qquad\tau_2^{\rm inc}-\tau_3^{\rm inc}=q_2\,,
	\end{align}
	which uniquely determines
	\begin{align}
		\matr{\tau_1^{\rm inc}\\\tau_2^{\rm inc}\\\tau_3^{\rm inc}}
		=\frac{1}{3}\matr{2 & -1\\-1 & 2\\	-1 & -1}q
		=\frac{1}{3}\matr{2\.\tau_1-\tau_2-\tau_3\\2\.\tau_2-\tau_1-\tau_3\\2\.\tau_3-\tau_1-\tau_2}
		=\matr{\tau_1\\\tau_2\\\tau_3}-\frac{\tau_1+\tau_2+\tau_3}{3}\matr{1\\1\\1}.
		\label{eq:intrinsic-from-reduced-energy}
	\end{align}
	Consequently,
	\begin{align}
		t_i^{\rm inc}=\frac{\tau_i^{\rm inc}}{\lambda_i}\,,
		\qquad i=1,2,3\,,
	\end{align}
	so that the intrinsic Biot-stress tensor is also completely determined by the reduced energy.
	
	We therefore obtain two different representatives of the same constrained stress response:
	\begin{align}
		p_{\rm inc}=\frac{\tau_1+\tau_2+\tau_3}{3}
		\qquad\text{and}\qquad
		p_{\rm ps}=\tau_3\,.
		\label{eq:two-pressure-choices}
	\end{align}
	The first choice yields the unique deviatoric intrinsic Kirchhoff stress, whereas the second yields a representative with vanishing stress in the distinguished third principal direction.
	Both representatives contain the same underlying information, namely the pressure-independent stress differences encoded by $\D\gtilde_{\rm red}$, but they are in general different tensor fields.
	
	Indeed, they differ only by a normal contribution,
	\begin{align}
		\TBiot^{\rm ps}(U)-\TBiot^{\rm inc}(U)
		=\bigl(p_{\rm inc}(U)-p_{\rm ps}(U)\bigr)U^{-1}.
		\label{eq:representative-difference}
	\end{align}
	Consequently,
	\begin{align}
		\iprod{\TBiot^{\rm ps}(U)-\TBiot^{\rm inc}(U),H}=0
		\qquad\forall\; H\in T_U\mathcal M\,.
	\end{align}
	Thus, both representatives describe exactly the same first variation of the incompressible energy.
	
	This equivalence at the level of the constrained first variation does not mean that arbitrary ambient monotonicity conditions for the two representatives are equivalent.
	In particular, for $U,\widetilde U\in\mathcal M$, the difference $U-\widetilde U$ is in general not an element of either $T_U\mathcal M$ or $T_{\widetilde U}\mathcal M$.
	Hence, the normal contributions in \eqref{eq:representative-difference} do not generally vanish in a two-point pairing of the form
	\begin{align}
		\iprod{T(U)-T(\widetilde U)\,,U-\widetilde U}.
	\end{align}
	For this reason, the intrinsic Biot-stress tensor and the plane-stress representative should be distinguished when formulating monotonicity conditions, even though both are uniquely reconstructed from the same reduced energy once the corresponding choice of representative has been prescribed.
	
	For completeness, the same construction can be formulated  for an isotropic energy $\widehat{\mathcal W}\col \Symp(n)\to\R$ in terms of the left stretch tensor $V$.
	If $F=R\.U=VR$, then
	\begin{align}
		V=R\.UR^T,
		\qquad\log V=R\.(\log U)R^T.
	\end{align}
	By isotropy and the Richter formula, cf.\ \cite{richter1949verzerrung,richter1948,agn_neff2020axiomatic},
	\begin{align}
		\D_{\log U}\Wtilde(\log U)
		=R^T\D_{\log V}\widetilde{\mathcal W}(\log V)R\,.
	\end{align}
	Thus, the intrinsic construction in terms of $U$ given above is equivalent to the corresponding formulation in terms of $V$.
	%
	%
	%
	\section{Comparison between the reduced Hessian and the parametrized Hessian}\setcounter{equation}{0}\label{se4}
	Let $\mathcal S\subset\R_+^3$ be the spectral incompressibility manifold from Definition \ref{def:incompressible-stretch-sets}, and let $g\col \R_+^3\to\R$ be a $C^2$-function denoting the representation of the energy in terms of the singular values of $U=\sqrt{F^T F}$.
	Define the constraint $s(\lambda_1,{\lambda_2},{\lambda_3})=\lambda_1\.{\lambda_2}\.{\lambda_3}-1 .$ Then $\D s(p)=({\lambda_2}\.{\lambda_3},\lambda_1\.{\lambda_3},\lambda_1\.{\lambda_2}).$ At $p=(\lambda_1,{\lambda_2},{\lambda_3})\in \mathcal{S}$, we have
	\begin{align}
		T_p\mathcal{S}
		=
		\left\{(u,v,w)\in\R^3\.|\D s(p).\.(u,v,w)=0\right\}
		=
		\left\{(u,v,w)\.\middle|\.
		\frac{u}{\lambda_1}+\frac{v}{{\lambda_2}}+\frac{w}{{\lambda_3}}=0\right\}.
	\end{align}
	A convenient basis of $T_p\mathcal{S}$ is
	\begin{align}
		b_1=(\lambda_1,-{\lambda_2},0)\,,
		\qquad
		b_2=(\lambda_1,0,-{\lambda_3})\,,
		\qquad
		B=[b_1\ b_2]\,.
	\end{align}
	A point $p\in \mathcal{S}$ is a critical point of $g|_ \mathcal{S}$ if and only if the directional
	derivative of $g$ vanishes along all tangent directions, i.e.
	\begin{align}
		\label{FOC}
		\D g(p).h=0 \qquad \forall\; h\in T_p \mathcal{S}\,.
	\end{align}
	This is equivalent to requiring that $\D g(p)$ is orthogonal to $T_p \mathcal{S}$,
	hence parallel to the normal vector $\left(\frac{1}{\lambda_1},\frac{1}{{\lambda_2}},\frac{1}{{\lambda_3}}\right).$ Therefore, there exists $\vartheta \in\R$ (the Lagrange multiplier associated with the constraint $s(\lambda_1,{\lambda_2},{\lambda_3})=0$) such that $\D g(p)=\vartheta\.\D s(p)$,
	i.e.\ the gradient of $g$ is normal to the constraint manifold. In the present
	case, using $\lambda_1{\lambda_2}{\lambda_3}=1$, one obtains $\frac{\partial g}{\partial \lambda_1}= \frac{\vartheta}{{\lambda_1}},\ \frac{\partial g}{\partial \lambda_2}= \frac{\vartheta}{{\lambda_2}},\ \frac{\partial g}{\partial \lambda_3}= \frac{\vartheta}{{\lambda_3}},$ which yields $\vartheta={\lambda_1} \.\frac{\partial g}{\partial \lambda_1}= {\lambda_2} \.\frac{\partial g}{\partial \lambda_2}= {\lambda_3} \. \frac{\partial g}{\partial \lambda_3}.$ The Hessian of the constraint is
	\begin{align}
		\D ^2 s(p)=
		\matr{
			0           & {\lambda_3} & {\lambda_2} \\
			{\lambda_3} & 0           & {\lambda_1} \\
			{\lambda_2} & {\lambda_1} & 0
		}.
	\end{align}
	At a point satisfying \eqref{FOC}, define the \emph{effective Hessian}
	\begin{align}
		H_{\mathrm{eff}}(p)
		=
		\D^2 g(p)
		-
		\bigl({\lambda_1} \frac{\partial g}{\partial \lambda_1}(p)\bigr)\D^2 s(p)\,.
	\end{align}
	The matrix $H_{\mathrm{eff}}$ represents the
	Hessian of the Lagrangian associated with the constrained problem. It accounts
	for both the curvature of the objective function and the curvature induced by
	the constraint.
	Indeed, let $\gamma(s)\subset \mathcal{S}$ be a smooth curve with $\gamma(0)=p$ and $\gamma'(0)=h\in T_p \mathcal{S}$. Since $s(\gamma(s))=0$, differentiation yields
	\begin{align}
		\iprod{\D s(p)\,,h}= 0\,,
	\end{align}
	and
	\begin{align}
		\iprod{\D^2 s(p).\. h, h} + \iprod{\D s(p)\,,\gamma''(0)}= 0\,.
		\qquad\iff\qquad
		\iprod{\D s(p)\,,\gamma''(0)}
		=
		- \iprod{\D^2 s(p).\. h, h}\,.
	\end{align}
	Moreover,
	\begin{align}
		\frac{{\rm d}^2}{{\rm d}s^2} g(\gamma(s))\big|_{s=0}
		=
		\iprod{\D^2 g(p).h, h} + \iprod{\D g(p)\,,\gamma''(0)}\,.
	\end{align}
	Using the first-order condition $\D g(p)=\vartheta \D s(p)$, we obtain
	\begin{align}
		\frac{{\rm d}^2}{{\rm d}s^2} g(\gamma(s))\big|_{s=0}
		=
		\iprod{\D^2 g(p).h,h}
		-
		\vartheta\.\iprod{\D^2 s(p).h, h}
	\end{align}
	which yields that for any tangent direction $h\in T_p \mathcal{S}$, one has
	\begin{align}
		\frac{{\rm d}^2}{{\rm d}s^2} g(\gamma(s))\big|_{s=0}
		=
		\iprod{H_{\mathrm{eff}}(p).h, h}
	\end{align}
	where $\gamma$ is any curve in $\mathcal{S}$ with $\gamma'(0)=h$.
	Thus, the restriction of $H_{\mathrm{eff}}$ to $T_p\mathcal S$ gives the second variation of the constrained energy and determines the nature of the critical point.
	The multiplier $\vartheta$ quantifies the influence of the constraint on the second-order behavior and appears naturally in the effective Hessian $H_{\mathrm{eff}}=\D^2 g-\vartheta\D^2 s$.
	Restricting $H_{\mathrm{eff}}$ to the tangent space yields the constrained Hessian
	governing local optimality.
	In logarithmic coordinates, where the constraint becomes linear, this corresponds to the standard Hessian of the reparametrized function.
	The Hessian restricted to tangent directions is
	\begin{align}
		H_{\mathrm{tan}}(p)
		=
		B^T H_{\mathrm{eff}}(p) B
		\in\R^{2\times2}.
	\end{align}
	We compare two second--order objects:
	\begin{itemize}
		\item the \emph{reduced (tangential) Hessian}
		arising from the constrained second--order condition;
		\item the Hessian of the parametrized function on the constraint
		\begin{align}
			g^{\rm red}({\lambda_1},{\lambda_2})=g\left({\lambda_1},{\lambda_2},\frac{1}{{\lambda_1}{\lambda_2}}\right).
		\end{align}
	\end{itemize}
	We show by explicit calculation that these matrices represent the \emph{same quadratic form}
	on the tangent space of $\mathcal{S}$.
	Define the parametrization
	\begin{align}
		\psi({\lambda_1},{\lambda_2})=\bigl(X,Y,Z\bigr)
		=
		\left({\lambda_1},\ {\lambda_2},\ {\lambda_3}({\lambda_1},{\lambda_2})\right),
		\qquad\qquad
		{\lambda_3}({\lambda_1},{\lambda_2})=\frac{1}{{\lambda_1}{\lambda_2}}\,.
	\end{align}
	The Jacobian $\widetilde J=\D\psi({\lambda_1},{\lambda_2})\in\R^{3\times2}$ has columns
	\begin{align}
		\frac{\partial \psi}{\partial \lambda_1}=\left(1,0,\frac{\partial \lambda_3}{\partial \lambda_1}\right),
		\qquad	\qquad
		\frac{\partial \psi}{\partial \lambda_2}=\left(0,1,\frac{\partial \lambda_3}{\partial \lambda_2}\right).
	\end{align}
	Since ${\lambda_3}={\lambda_1}^{-1}{\lambda_2}^{-1}$, $\frac{\partial \lambda_3}{\partial \lambda_1}=-\frac{{\lambda_3}}{{\lambda_1}}, \frac{\partial \lambda_3}{\partial \lambda_2}=-\frac{{\lambda_3}}{{\lambda_2}},$ we have
	\begin{align}
		\widetilde J=\D \psi=
		\matr{
			1                                & 0                                \\
			0                                & 1                                \\
			-\frac{{\lambda_3}}{{\lambda_1}} & -\frac{{\lambda_3}}{{\lambda_2}}
		}.
	\end{align}
	Using $g^{\rm red}({\lambda_1},{\lambda_2})=g\circ\psi({\lambda_1},{\lambda_2})$, by the chain rule for second derivatives,
	\begin{align}
		\D^2 g^{\rm red}
		=
		\widetilde{J}^T (\D^2 g) \widetilde J
		+
		\frac{\partial g}{\partial \lambda_1}\D^2 X
		+
		\frac{\partial g}{\partial \lambda_2}\D^2 Y
		+
		\frac{\partial g}{\partial \lambda_3}\D^2 Z\,.
	\end{align}
	Since $X={\lambda_1}$ and $Y={\lambda_2}$ are linear, $\D^2 X=\D^2 Y=0,$ and therefore
	\begin{equation}
		\D^2 g^{\rm red}
		=
		\widetilde J^T (\D^2 g) \widetilde J
		+
		\frac{\partial g}{\partial \lambda_3}\D^2 {\lambda_3}\,.
	\end{equation}
	We compute $\frac{\partial }{\partial \lambda_1}{\lambda_3}=-\frac{{\lambda_3}}{{\lambda_1}}$, $\frac{\partial }{\partial\lambda_2}{\lambda_3}=-\frac{{\lambda_3}}{{\lambda_2}}$ and the
	second derivatives $\frac{\partial^2 }{\partial \lambda_1^2}{\lambda_3}=\frac{2{\lambda_3}}{{\lambda_1}^2}$, $\frac{\partial^2 }{\partial \lambda_2^2}{\lambda_3}=\frac{2{\lambda_3}}{{\lambda_2}^2}$, $\frac{\partial^2 }{\partial \lambda_1\lambda_2}{\lambda_3}=\frac{{\lambda_3}}{{\lambda_1}{\lambda_2}}.$ Hence, using $\D^2 {\lambda_3} = \matr{ \frac{2{\lambda_3}}{{\lambda_1}^2} & \frac{{\lambda_3}}{{\lambda_1}{\lambda_2}}\\ \frac{{\lambda_3}}{{\lambda_1}{\lambda_2}} & \frac{2{\lambda_3}}{{\lambda_2}^2} },$ and $\widetilde J^T(\D^2 s)\widetilde J = \matr{-\frac{2}{{\lambda_1}^2} & -\frac{1}{{\lambda_1}{\lambda_2}}\\ -\frac{1}{{\lambda_1}{\lambda_2}} & -\frac{2}{{\lambda_2}^2}}$ we have $\D^2 {\lambda_3} = -{\lambda_3}\. \widetilde J^T(\D^2 s)\widetilde J.$
	We recall that a point $p=({\lambda_1},{\lambda_2},{\lambda_3})\in \mathcal{S}$ is a constrained critical point of $g|_\mathcal{S}$ if and only if
	\begin{align}
		{\lambda_1} \frac{\partial g}{\partial \lambda_1}(p)={\lambda_2} \frac{\partial g}{\partial \lambda_2}(p)={\lambda_3} \frac{\partial g}{\partial \lambda_3}(p)=\vartheta\,.
	\end{align}
	In particular, ${\lambda_3} \frac{\partial g}{\partial \lambda_3} = \vartheta$. Therefore, we deduce
	\begin{align}
		\D^2 g^{\rm red}
		=
		\widetilde J^T (\D^2 g) \widetilde J
		-
		\vartheta\.\widetilde J^T(\D^2 s) \widetilde J
		=
		\widetilde J^T\bigl(\D^2 g-\vartheta\D^2 s\bigr)\widetilde J
		=
		\widetilde J^T H_{\mathrm{eff}} \widetilde J
	\end{align}
	using the \emph{effective Hessian} $H_{\mathrm{eff}}$.
	Since the columns of $B$ and $\widetilde J$ span the same tangent space, there exists an
	invertible matrix $S\in\R^{2\times2}$ such that $B=\widetilde J S.$ Therefore,
	\begin{align}
		H_{\mathrm{tan}}
		=
		S^T (\widetilde J^T H_{\mathrm{eff}} \widetilde J) S
		=
		S^T (\D^2 g^{\rm red}) S\,.
	\end{align}
	Hence $H_{\mathrm{tan}}$ and $\D^2 g^{\rm red}$ are \emph{congruent} matrices and
	represent the same quadratic form.
	At a constrained critical point $p\in\mathcal S$, we therefore have
	\begin{align}
		H_{\mathrm{tan}}\in \Symp(2)
		\quad\iff\quad
		\D^2 g^{\rm red}\in \Symp(2)\,.
	\end{align}
	Thus, positivity of the constrained Hessian is independent of the chosen tangent basis.

	\section{Geodesic convexity on the incompressible stretch manifold}\label{geoc}\setcounter{equation}{0}
	Let $\mathcal S\subset\R_+^3$ be the spectral incompressibility manifold from Definition \ref{def:incompressible-stretch-sets}. It is a two-dimensional smooth submanifold of $\R_+^3$. Its tangent space at
	$\lambda\in\mathcal S$ is given by
	$T_\lambda\mathcal S = \left\{ \xi\in\R^3 \;|\; \sum_{i=1}^3 \frac{\xi_i}{\lambda_i}=0 \right\}.$
	We introduce logarithmic coordinates $x_i=\log \lambda_i, i=1,2,3.$ Then the incompressibility constraint becomes $x_1+x_2+x_3=0.$ Hence $\mathcal S$ is globally parametrized by the linear plane $\Pi \colonequals \left\{ x\in\R^3 \;:\; x_1+x_2+x_3=0 \right\}$ through the componentwise exponential map
	\begin{align}
		\exp\col \Pi\to\mathcal S,
		\qquad
		x\mapsto
		\bigl(e^{x_1},e^{x_2},e^{x_3}\bigr).
	\end{align}
	Conversely,
	\begin{align}
		\log\col \mathcal S\to\Pi,
		\qquad
		\lambda\mapsto
		\bigl(\log\lambda_1,\log\lambda_2,\log\lambda_3\bigr).
	\end{align}
	We endow $\mathcal S$ with the Riemannian metric obtained by pulling back the Euclidean
	metric of $\Pi$ through the logarithmic map. Thus, for
	$\xi,\eta\in T_\lambda\mathcal S$, we set
	\begin{align}
		\mathfrak g_\lambda(\xi,\eta)
		\colonequals
		\iprod{\D\log_\lambda[\xi],
			\D\log_\lambda[\eta]}_{\R^3}
		=
		\sum_{i=1}^3
		\frac{\xi_i\.\eta_i}{\lambda_i^2}.
	\end{align}
	With respect to this logarithmic, or log-Euclidean, Riemannian structure, the map
	\begin{align}
		\log\col (\mathcal S,\mathfrak g)\to \Pi
	\end{align}
	is an isometry.
	Since $\log$ is an isometry onto the flat plane $\Pi$, the geodesics on
	$(\mathcal S,\mathfrak g)$ are precisely the images under the exponential map of
	straight lines in $\Pi$. Therefore, for $\lambda,\mu\in\mathcal S$, the geodesic
	joining $\lambda$ to $\mu$ is
	\begin{align}
		\gamma_{\lambda,\mu}(s)
		=
		\exp\bigl((1-s)\log\lambda+s\log\mu\bigr)\,,
		\qquad s\in[0,1]\,.
	\end{align}
	Equivalently, componentwise,
	\begin{align}
		\gamma_{\lambda,\mu}(s)
		=
		\left(
		\lambda_1^{1-s}\mu_1^s,
		\lambda_2^{1-s}\mu_2^s,
		\lambda_3^{1-s}\mu_3^s
		\right).
	\end{align}
	This curve remains in $\mathcal S$, since
	\begin{align}
		\gamma_1(s)\gamma_2(s)\gamma_3(s)
		=
		(\lambda_1\lambda_2\lambda_3)^{1-s}
		(\mu_1\mu_2\mu_3)^s
		=
		1.
	\end{align}
	Let $g^{\rm inc}\col \mathcal S\to\R$ be a reduced incompressible energy written in terms of
	the principal stretches. We say that $g^{\rm inc}$ is geodesically convex on $\mathcal S$, with
	respect to the logarithmic Riemannian structure, if
	\begin{align}
		g^{\rm inc}(\gamma_{\lambda,\mu}(s))
		\leq
		(1-s)\.g^{\rm inc}(\lambda)+s\. g^{\rm inc}(\mu)
	\end{align}
	for all $\lambda,\mu\in\mathcal S$ and all $s\in[0,1]$. Equivalently,
	\begin{align}
		g^{\rm inc}\left(
		\lambda_1^{1-s}\mu_1^s,
		\lambda_2^{1-s}\mu_2^s,
		\lambda_3^{1-s}\mu_3^s
		\right)
		\leq
		(1-s)\.g^{\rm inc}(\lambda)+s\.g^{\rm inc}(\mu).
	\end{align}
	Define
	\begin{align}
		\gtilde^{\rm inc}\col \Pi\to\R,
		\qquad
		\gtilde^{\rm inc}(x)
		\colonequals
		g^{\rm inc}(e^{x_1},e^{x_2},e^{x_3}).
	\end{align}
	Then
	\begin{align}
		g^{\rm inc} \text{ is geodesically convex on } \mathcal S
		\quad\iff\quad
		\gtilde^{\rm inc} \text{ is convex on } \Pi.
	\end{align}
	Indeed, the geodesic segment on $\mathcal S$ is exactly the exponential image of the
	Euclidean segment in $\Pi$. Hence
	\begin{align}
		\gtilde^{\rm inc}((1-s)\.x+s\.y)
		\leq
		(1-s)\.\gtilde^{\rm inc}(x)+s\.\gtilde^{\rm inc}(y)
	\end{align}
	is precisely
	\begin{align}
		g^{\rm inc}\bigl(\exp((1-s)\.x+s\. y)\bigr)
		\leq
		(1-s)\.g^{\rm inc}(e^x)+s\. g^{\rm inc}(e^y).
	\end{align}
	Putting $\lambda=e^x$ and $\mu=e^y$, this becomes the preceding geodesic convexity
	condition.
	Equivalently, one may use two independent variables. Since $x_3=-x_1-x_2,$ we define
	\begin{align}
		\gtilde_{\rm red}\col \R^2\to\R,
		\qquad
		\gtilde_{\rm red}(x_1,x_2)
		\colonequals
		g^{\rm inc}(e^{x_1},e^{x_2},e^{-x_1-x_2}).
	\end{align}
	Then
	\begin{align}
		g^{\rm inc} \text{ is geodesically convex on } \mathcal S
		\quad\iff\quad
		\gtilde_{\rm red} \text{ is convex on } \R^2.
	\end{align}
	If $g^{\rm inc}$ is continuously differentiable, the convexity of $\gtilde^{\rm inc}$ on the linear
	space $\Pi$ is equivalent to the monotonicity of its derivative, namely
	\begin{align}
		\iprod{\D\gtilde^{\rm inc}(x)-\D\gtilde^{\rm inc}(y)\,,
			x-y}
		\geq 0
		\qquad
		\forall\;x,y\in\Pi.
	\end{align}
	To express this tangential derivative by ambient partial derivatives, let $g$ be any $C^1$ extension of $g^{\rm inc}$ to a neighbourhood of $\mathcal S$ in $\R_+^3$. Writing $\lambda=e^x$, we obtain, for every $h\in\Pi$,
	\begin{align}
		\D\gtilde^{\rm inc}(x)[h]
		=
		\sum_{i=1}^3
		\lambda_i\.\frac{\partial g}{\partial \lambda_i}(\lambda)\.h_i.
	\end{align}
	Thus the derivative of $\gtilde^{\rm inc}$ is represented, on the tangent plane $\Pi$, by
	the vector
	\begin{align}
		\tau^{\rm amb}(\lambda)
		\colonequals
		\left(
		\lambda_1 \frac{
			\partial} {\partial \lambda_1}g(\lambda)\,,
		\lambda_2 \frac{
			\partial} {\partial \lambda_2}g(\lambda)\,,
		\lambda_3 \frac{
			\partial} {\partial \lambda_3}g(\lambda)
		\right).
	\end{align}
	Two such extensions change $\tau^{\rm amb}(\lambda)$ only by a multiple of $(1,1,1)^T$, so the tangential derivative and the monotonicity pairing on $\Pi$ are independent of the chosen extension.
	This is the principal Kirchhoff stress vector modulo the arbitrary hydrostatic pressure.
	Indeed, in the incompressible case the principal Kirchhoff stresses have the form
	\begin{align}
		\tau_i(\lambda)
		=
		\lambda_i \frac{
			\partial} {\partial \lambda_i}g(\lambda)-p,
		\qquad i=1,2,3,
	\end{align}
	where $p\in\R$ is the Lagrange multiplier associated with the constraint
	$\lambda_1\lambda_2\lambda_3=1$. Since $x-y\in\Pi$, we have
	\begin{align}
		\iprod{x-y,(1,1,1)^T}=0.
	\end{align}
	Consequently, the pressure term does not contribute to the monotonicity inequality.
	Therefore, geodesic convexity of $g^{\rm inc}$ on $\mathcal S$, with respect to the logarithmic
	structure, is equivalent to
	\begin{align}
		\sum_{i=1}^3
		\bigl(\tau_i(\lambda)-\tau_i(\mu)\bigr)
		\bigl(\log\lambda_i-\log\mu_i\bigr)
		\geq 0
		\qquad
		\forall\;\lambda,\mu\in\mathcal S.
	\end{align}
	Equivalently,
	\begin{align}
		\iprod{\dev\tau(\lambda)
			-
			\dev\tau(\mu)\,,
			\log\lambda-\log\mu}
		\geq 0
		\qquad
		\forall\;\lambda,\mu\in\mathcal S.
	\end{align}
	Thus, in the incompressible setting, the geodesic convexity of the reduced energy in
	logarithmic stretch variables is precisely the monotonicity of the tangential, or
	deviatoric, part of the Kirchhoff stress with respect to the logarithmic stretches, which is Hill's inequality in the incompressible case \cite[Appendix-D3-Neff]{baaser2026hyperelastic}.

	\section{Alternative proof of the non-strict spectral monotonicity implication}\label{app:alternative-spectral-proof}
	We record here the argument contained in an earlier version of this manuscript. It gives a self-contained proof of the non-strict implication from vector monotonicity to matrix monotonicity, but it is not needed for the main development because Theorem \ref{theorem:mainResult} is now available from \cite{agn_voss2026monotonicity}. The argument does not by itself settle the strict equality case when two distinct matrices are isospectral; the additional equality-case analysis required for strict monotonicity is given in \cite{agn_voss2026monotonicity}.
	\begin{proof}[Alternative proof of the non-strict implication in Theorem \ref{theorem:mainResult}] The first step in the proof of this theorem is to show that
		for every $x\in M\subset\R^n$ and all $i\neq j$,
		\begin{align}
			(x_i-x_j)\.(f_i(x)-f_j(x)) \geq 0.
		\end{align}
		In particular, $x$ and $f(x)$ are sorted in the same algebraic order.
		Fix $x\in M$ and indices $i\neq j$.
		Let $\mathcal{P}^{(i,j)}$ be the elementary permutation matrix which swaps the
		components $i$ and $j$ of a vector.
		Define $y=\mathcal{P}^{(i,j)}x,$ i.e.\ $y$ is obtained from $x$ by exchanging the components $x_i$ and $x_j$ and since $M$ is permutation-invariant, we have $y\in M$.
		Equivalently, if $\pi$ denotes the permutation exchanging $i$ and $j$, then $y=\pi(x).$ By vector monotonicity,
		\begin{align}
			0\leq \iprod{f(x)-f(y)\,,\.x-y} .
		\end{align}
		By permutation equivariance we have $f_i(y)=f_i(x_{\pi(1)},x_{\pi(2)},\dotsc,x_{\pi(n)}) = f_{\pi(i)}(x_1,\dotsc,x_n)$, i.e.\ in terms of the permutation matrix $\mathcal{P}^{(i,j)}$, $f(y)=\mathcal{P}^{(i,j)}f(x)$.
		If $x_i=x_j$, permutation equivariance also implies $f_i(x)=f_j(x)$.
		Hence
		\begin{align}
			0\leq \iprod{f(x)-\mathcal{P}^{(i,j)}f(x)\,,\.x-\mathcal{P}^{(i,j)}x} .
		\end{align}
		Only components $i$ and $j$ contribute since the components which are not permuted lead to zero entries in the differences appearing in the above inner product, giving
		\begin{align}
			0\leq 2\.(f_i(x)-f_j(x))(x_i-x_j).
		\end{align}
		This proves the claim.
		As a consequence, if $x^\downarrow$ denotes $x$ sorted in decreasing order, then
		\begin{align}
			\label{forder}
			f(x)^\downarrow = f(x^\downarrow).
		\end{align}
		The second step is the reduction of the matrix monotonicity to a diagonal--orthogonal form. Let
		\begin{align}
			U=Q^T\.\diag(\lambda)\.Q,\qquad \widetilde{U}=R^T\.\diag(\mu)\.R.
		\end{align}
		Using orthogonal invariance of the trace,
		\begin{align}
			\iprod{\Sigma_f(U)-\Sigma_f(\widetilde{U})\,,\.U-\widetilde{U}}
			=
			\iprod{\diag(f(\lambda))-\widetilde{Q} \diag(f(\mu)) \widetilde{Q} ^T, \diag(\lambda)- \widetilde{Q} \diag(\mu) \widetilde{Q} ^T}
		\end{align}
		where $\widetilde{Q} \colonequals Q\. R^T \in\OO(n)$.
		Thus it suffices to prove
		\begin{equation}
			\label{eq:reduced_problem}
			I\colonequals \iprod{\diag(f(\lambda))-\widetilde{Q} \diag(f(\mu)) \widetilde{Q} ^T, \diag(\lambda)- \widetilde{Q} \diag(\mu) \widetilde{Q} ^T}\geq 0
			\quad \forall\;\widetilde{Q}\in\OO(n)\,.
		\end{equation}
		Expanding \eqref{eq:reduced_problem}, we have
		\begin{align}
			\label{eq:reduced_problem2}
			I = &
			\iprod{\diag(f(\lambda))\,,\diag(\lambda)}
			+\iprod{\diag(f(\mu))\,,\diag(\mu)}\notag                                \\
			& -\iprod{\diag(f(\lambda))\widetilde{Q}, \widetilde{Q}\diag(\mu)}
			-\iprod{\diag(\lambda) \widetilde{Q} , \widetilde{Q}\diag(f(\mu))}.
		\end{align}
		Then, in another step we establish an orthogonal mixing and rearrangement inequality. For this, we remark that for given two vectors $a,b\in\R^n$ and $\widetilde{Q}=(\widetilde{Q}_{ij})\in\OO(n)$,
		\begin{equation}
			\label{eq:trace_identity}
			\iprod{\diag(a) \widetilde{Q}, \widetilde{Q}\diag(b)}= \sum_{i,j=1}^n a_i \.b_j\. \widetilde{Q}_{ij}^2=\iprod{a, D.b},
		\end{equation}
		where $D=(d_{ij})$ is defined by $d_{ij}=\widetilde{Q}_{ij}^2$.
		The matrix $D$ is doubly stochastic, i.e.\ it satisfies
		\begin{align}
			d_{ij}\geq 0,\qquad \qquad
			\sum_{j=1}^n d_{ij}=1,\qquad \qquad
			\sum_{i=1}^n d_{ij}=1,
		\end{align}
		since the orthogonality of $\widetilde{Q}$ is exactly what guarantees that the row vectors and the column vectors have norm 1 and since
		squaring the entries converts those norm-1 conditions into the row and column sum conditions for a doubly stochastic matrix.
		By the Birkhoff--von Neumann theorem \cite{birkhoff1946three,hardy1952inequalities,horn1954doubly}, every doubly stochastic matrix
		can be written as a convex combination of permutation matrices.
		Hence there exist permutation matrices $\mathcal{P}_k$, $k=1,\dotsc,N$, and coefficients ${c_k} \geq 0$, $k=1,\dotsc,N,$ with $\sum_{k=1}^N {c_k} = 1$ such that
		\begin{align}
			D = \sum_{k=1}^N {c_k} \mathcal{P}_k.
		\end{align}
		The map $D \mapsto\iprod{a, D.b}$ is linear in $D$. Therefore,
		\begin{align}
			\iprod{a, D.b}
			=
			\iprod{a, \left( \sum_{k=1}^N {c_k} \mathcal{P}_k \right)\!\!.b}
			=
			\sum_{k=1}^N {c_k} \. \iprod{a, \mathcal{P}_k. b}.
		\end{align}
		Thus $\iprod{a, D.b}$ is a convex combination of the numbers $\iprod{a, \mathcal{P}_k b}$. Since ${c_k} \geq 0$ and $\sum_{k=1}^N {c_k} = 1$,
		the above expression is a convex combination.
		A convex combination cannot exceed the largest term in the list.
		Hence
		\begin{align}
			\iprod{a, D.b}
			\leq
			\max_{\mathcal{P} \ \mathrm{permutation \ matrix}}\iprod{a, \mathcal{P}. b}
		\end{align}
		\medskip
		If $P$ corresponds to a permutation $\pi$, then $\iprod{a, P. b} = \sum_i a_i\. b_{\pi(i)}.$ We recall the rearrangement inequality \cite{hardy1952inequalities,MarshallOlkinArnold}.
		Let $a_1 \leq a_2 \leq \dots \leq a_n, b_1 \leq b_2 \leq \dots \leq b_n$
		be two real sequences ordered increasingly.
		Then for every permutation $\pi$ of $\{1,\dots,n\}$ we have
		\begin{align}
			\sum_{i=1}^{n} a_i \.b_{n-i+1}
			\;\leq\;
			\sum_{i=1}^{n} a_i \.b_{\pi(i)}
			\;\leq\;
			\sum_{i=1}^{n} a_i\. b_i .
		\end{align}
		The maximum of the sum is attained when the sequences are ordered in the
		same way, i.e.
		\begin{align}
			\max_{\pi}\sum_{i=1}^{n} a_i \.b_{\pi(i)}
			=
			\sum_{i=1}^{n} a_i \.b_i ,
		\end{align}
		and the minimum is attained when they are ordered oppositely:
		\begin{align}
			\min_{\pi}\sum_{i=1}^{n} a_i\. b_{\pi(i)}
			=
			\sum_{i=1}^{n} a_i \.b_{n-i+1}.
		\end{align}
		For our vectors $a$ and $b$, let $a^\downarrow$ and $b^\downarrow$ denote the vectors
		obtained by sorting $a$ and $b$ in decreasing order.
		Then
		\begin{align}
			\max_{\mathcal{P} \ \mathrm{perm}} \iprod{a, \mathcal{P}. b}
			=
			\iprod{a^\downarrow, b^\downarrow}.
		\end{align}
		Combining the above steps, for all $\widetilde{Q}\in\OO(n)$, it holds true
		\begin{align}
			\iprod{\diag(a) \widetilde{Q}, \widetilde{Q} \diag(b)}
			=
			\iprod{a, D.b}
			\leq
			\max_{\mathcal{P} \ \mathrm{perm}} \iprod{a, \mathcal{P}. b}
			=
			\iprod{a^\downarrow, b^\downarrow}.
		\end{align}
		Thus
		\begin{equation}
			\label{eq:rearrangement}
			\iprod{\diag(a) \widetilde{Q}, \widetilde{Q}\diag(b)} \leq \iprod{a^\downarrow, b^\downarrow}.
		\end{equation}
		Using \eqref{eq:reduced_problem2} and \eqref{eq:rearrangement}, we deduce
		\begin{align}
			I \geq
			\iprod{f(\lambda)\,,\lambda} + \iprod{f(\mu)\,,\mu}
			- \iprod{f(\lambda)^\downarrow,\mu^\downarrow}
			- \iprod{\lambda^\downarrow,f(\mu)^\downarrow}.
		\end{align}
		By \eqref{forder}, we have $f(\lambda)^\downarrow = f(\lambda^\downarrow)\,, \ f(\mu)^\downarrow = f(\mu^\downarrow)\,,$ and $\iprod{f(\lambda)\,,\lambda} =\iprod{f(\lambda^\downarrow)\,,\lambda^\downarrow}.$ Therefore,
		\begin{align}
			I \geq
			\iprod{f(\lambda^\downarrow)-f(\mu^\downarrow)\,,\.
				\lambda^\downarrow-\mu^\downarrow}.
		\end{align}
		By vector monotonicity we have
		\begin{align}
			\iprod{f(\lambda^\downarrow)-f(\mu^\downarrow)\,,\.
				\lambda^\downarrow-\mu^\downarrow} \geq 0,
		\end{align}
		hence $I$ is nonnegative and the matrix monotonicity follows. {The reverse implication follows immediately by considering the special case of diagonal matrices.}\qedhere
	\end{proof}
\end{document}